\documentclass[a4paper,12pt]{article}
\usepackage{amsmath,amsthm,amssymb}
\usepackage{amsfonts}
\usepackage{mathrsfs}

\advance \textheight by \topskip
\advance \textheight by 1pt
\makeatother

\numberwithin{equation}{section}

\newtheorem{thm}{Theorem}[section]
\newtheorem{lem}[thm]{Lemma}

\newtheorem{prop}[thm]{Proposition}

 {\theoremstyle{definition}
\newtheorem{defn}[thm]{Definition}

\newtheorem{rem}[thm]{Remark}
\newtheorem{ex}[thm]{Example}}

\def\l{\lambda}

\def\l{\lambda}

\def\b{\begin{equation}}
\def\e{\end{equation}}

\def\la{\label}
\def\non{\nonumber}

\title{Hurwitz integrality of the power series expansion of the sigma function for a telescopic curve}
\author{Takanori Ayano\footnote{Osaka Central Advanced Mathematical Institute, Osaka Metropolitan University, \newline \hspace{3ex} 3-3-138, Sugimoto, Sumiyoshi-ku, Osaka, 558-8585, Japan. \newline \hspace{3ex} Email: ayano@omu.ac.jp
\newline \hspace{3ex} 2020 Mathematics Subject Classification. Primary 14H42; Secondary 14K25, 32A05. 
\newline \hspace{3ex} Key Words and Phrases. sigma function, telescopic curve, power series expansion, Hurwitz integrality.
}}

\date{}

\begin{document}
\maketitle

\begin{abstract}
A telescopic curve is a certain algebraic curve defined by $m-1$ equations in the affine space of dimension $m$, which can be a hyperelliptic curve and an $(n,s)$ curve as a special case. 
The sigma function $\sigma(u)$ associated with the telescopic curve of genus $g$ is a holomorphic function on $\mathbb{C}^g$. 
For a subring $R$ of $\mathbb{C}$ and variables $u={}^t(u_1,\dots, u_g)$, let
\[R\langle\langle u \rangle\rangle=\left\{\sum_{i_1,\dots,i_g\ge0}\kappa_{i_1,\dots,i_g}\frac{u_1^{i_1}\cdots u_g^{i_g}}{i_1!\cdots i_g!}\;\middle|\;\kappa_{i_1,\dots,i_g}\in R\right\}.\]
If the power series expansion of a holomorphic function $f(u)$ on $\mathbb{C}^g$ around the origin belongs to $R\langle\langle u \rangle\rangle$, then $f(u)$ is said to be Hurwitz integral over $R$. 
In this paper, we show that the sigma function $\sigma(u)$ associated with the telescopic curve is Hurwitz integral over the ring generated by the coefficients of the defining equations of the curve and $\frac{1}{2}$ over $\mathbb{Z}$. 
Further, we show that $\sigma(u)^2$ is Hurwitz integral over the ring generated by the coefficients of the defining equations of the curve over $\mathbb{Z}$. 
Our results are a generalization of the results of Y. \^Onishi for $(n,s)$ curves to telescopic curves. 
\end{abstract}

\section{Introduction}

The Weierstrass's elliptic sigma function plays important roles in the theory of the Weierstrass's elliptic function.  
F.~Klein \cite{Kl1,Kl2} generalized the Weierstrass's elliptic sigma function to the multivariate sigma functions associated with hyperelliptic curves. 
V. M. Buchstaber, V. Z. Enolski, and D. V. Leykin improved the theory of the Klein's hyperelliptic sigma functions and generalized it to more general plane algebraic curves called $(n,s)$ curves 
(e.g., \cite{BEL-97-1, BEL-97-2, BEL-99-R, BEL-2000, BEL-2012, BEL-2018, BL-2005, BEL-99-2, EEL}). 
The sigma function is obtained by modifying the Riemann's theta function so as to be modular invariant, i.e., it does not depend on the choice of a canonical homology basis. 
Further, the sigma function has some remarkable algebraic properties that it is directly related with the defining equations of an algebraic curve. 
Namely, the coefficients of the power series expansion of the sigma function around the origin become polynomials of the coefficients of the defining equations of the algebraic curve. 
This property is important in the study of differential structure of Abelian functions (cf. \cite{CN, N-2011}). 
Further, from this property of the sigma function, the sigma function has a limit when the coefficients of the defining equations of the curve are specialized in any way, which is important in the study of integrable systems (cf. \cite{BEN2020, N-2018}). 
It is one of the central problems to determine the coefficients of the power series expansion of the sigma function. 
This problem is studied in many papers (e.g., \cite{Baker2, Bolza0, Bolza, BL-2005, O5, EO2019, N1, N-2010-2, O-2018}). 



Throughout the present paper, we denote by $\mathbb{N}, \mathbb{Z}_{\ge0}, \mathbb{Z}, \mathbb{Q}$, and $\mathbb{C}$ the sets of positive integers, non-negative integers, integers, rational numbers, and complex numbers, respectively.
For a subring $R$ of $\mathbb{C}$ and a set of some complex numbers $\mathcal{A}$, we denote by $R[\mathcal{A}]$ the ring generated by elements in $\mathcal{A}$ over $R$. 
For positive integers $k_1,\dots,k_n$, let $\langle k_1,\dots, k_n \rangle=\{\ell_1k_1+\cdots+\ell_nk_n\:|\:\ell_1,\dots, \ell_n\in\mathbb{Z}_{\ge0}\}$ and we denote by $\mathrm{gcd}(k_1,\dots,k_n)$ the greatest common divisor of $k_1,\dots,k_n$. 
For a subring $R$ of $\mathbb{C}$ and variables $z={}^t(z_1,\dots, z_n)$, let
\[R\langle\langle z \rangle\rangle=R\langle\langle z_1,\dots,z_n\rangle\rangle=\left\{\sum_{i_1,\dots,i_n\ge0}\kappa_{i_1,\dots,i_n}\frac{z_1^{i_1}\cdots z_n^{i_n}}{i_1!\cdots i_n!}\;\middle|\;\kappa_{i_1,\dots,i_n}\in R\right\}.\]
If the power series expansion of a holomorphic function $f(z)=f(z_1,\dots,z_n)$ on $\mathbb{C}^n$ around the origin belongs to $R\langle\langle z \rangle\rangle$, then we write $f(z)\in R\langle\langle z \rangle\rangle$ and $f(z)$ is said to be \textit{Hurwitz integral} over $R$. 
For a ring $R$ and a positive integer $n$, let $M_n(R)$ be the set of the $n\times n$ matrices such that all the components are contained in $R$. 
For a semigroup $Z$ and a positive integer $n$, let $nZ=\{nr\;|\;r\in Z\}$. 
We denote by $\emptyset$ the empty set.

In \cite{Miu}, Miura introduced a certain canonical form, Miura canonical form, for defining equations of any non-singular algebraic curve. 
A telescopic curve \cite{Miu} is a special curve for which Miura canonical form is easy to determine. 
For an integer $m\ge 2$, let $A_m=(a_1, \dots, a_m)$ be a sequence of positive integers
such that $\mathrm{gcd}(a_1, \dots, a_m)=1$, $a_i\ge2$ for any $i$, and 
\[
\frac{a_i}{d_i}\in \left\langle\frac{a_1}{d_{i-1}}, \dots, \frac{a_{i-1}}{d_{i-1}}\right\rangle,
\quad
2\le i\le m,
\]
where $d_i=\mathrm{gcd}(a_1, \dots, a_i)$. 
Let
\[
B(A_m)=\left\{(\ell_1, \dots, \ell_m)\in {\mathbb Z}_{\geq 0}^m\,\middle|\,0\le \ell_i\le \frac{d_{i-1}}{d_i}-1\,\, \text{ for }\,\,2\le i\le m\right\}.
\]
For any $2\le i\le m$, we have the unique element 
$(\ell_{i,1}, \dots, \ell_{i,m})\in B(A_m)$ satisfying 
\[\sum_{j=1}^m a_j \ell_{i,j}= a_i\frac{d_{i-1}}{d_i}.\]
For any $2\le i\le m$, we have $\ell_{i,j}=0$ for $j\geq i$. 
Consider the $m-1$ polynomials in $m$ variables $X=(X_1,\dots, X_m)$ given by
\[
F_i(X)=X_i^{d_{i-1}/d_i}-\prod_{j=1}^{i-1} X_j^{\ell_{i,j}}-\sum \lambda_{a_id_{i-1}/d_i-\sum_{k=1}^m a_kj_k}^{(i)}X_1^{j_1}\cdots X_m^{j_m},
\]
where $2\leq i\leq m$, $\lambda_{a_id_{i-1}/d_i-\sum_{k=1}^m a_kj_k}^{(i)}\in\mathbb{C}$, 
and the summation is taken over 
$(j_1,\dots, j_m)\in B(A_m)$ such that 
\[
\sum_{k=1}^m a_kj_k<a_i\frac{d_{i-1}}{d_i}.
\]
The algebraic curve defined by the $m-1$ equations $F_i=0$, $2\le i\le m$, in ${\mathbb C}^m=(X_1,\dots,X_m)$ is called the telescopic curve associated with $A_m$. 
We denote by ${\boldsymbol \lambda}$ the set of all $\lambda^{(i)}_j$. 
For a subset $\mathfrak{S}$ of ${\boldsymbol \lambda}$, 
we set $\lambda_{j, \mathfrak{S}}^{(i)}=\lambda_j^{(i)}/2$ if $\lambda_j^{(i)}\in \mathfrak{S}$ and we set $\lambda_{j, \mathfrak{S}}^{(i)}=\lambda_j^{(i)}$ 
if $\lambda_j^{(i)}\notin \mathfrak{S}$. 
We denote by ${\boldsymbol \l}_{\mathfrak{S}}$ the set of all $\lambda_{j, \mathfrak{S}}^{(i)}$. 
The sigma function $\sigma(u)$ associated with the telescopic curve of genus $g$ is a holomorphic function on $\mathbb{C}^g$. 
For $m=2$ and $A_2=(n,s)$ with relatively prime positive integers $n$ and $s$ such that $2\le n<s$, it is called the $(n,s)$ curve (cf. \cite{BEL-99-R}). 
For $m=2$, for simplicity, we denote $\lambda_j^{(2)}$ by $\lambda_j$. 
The polynomial $F_2$ defining the $(n,s)$ curve is given by 
\begin{equation}
F_2(X)=X_2^n-X_1^s-\sum\lambda_{ns-nj_1-sj_2}X_1^{j_1}X_2^{j_2},\label{2024.2.1.1114}
\end{equation}
where the summation is taken over $(j_1,j_2)\in{\mathbb Z}_{\geq 0}^2$ such that $nj_1+sj_2<ns$. 
Note that $(2,s)$ curves are none other than hyperelliptic curves.

We consider the $(n,s)$ curve. 
In \cite{N1}, an expression of the sigma function associated with the $(n,s)$ curve in terms of algebraic functions and integrals of an algebraic differential form is derived. 
In \cite{N-2010-2}, an expression of the sigma function associated with the $(n,s)$ curve in terms of the tau function of the KP-hierarchy is derived. 
In \cite{N1,N-2010-2}, by using these expressions of the sigma function of the $(n,s)$ curve, it is proved that we have $\sigma(u)\in\mathbb{Q}[{\boldsymbol \lambda}]\langle\langle u \rangle\rangle$ for the $(n,s)$ curve. 
Let $\mathfrak{C}$ be the set of $\lambda_j$ which is the coefficient of $X_1^{j_1}X_2^{j_2}$ with $(j_1,j_2)=(\mbox{odd}, \mbox{odd})$ in (\ref{2024.2.1.1114}).  
In \cite{O-2018}, a special local parameter of the $(n,s)$ curve around $\infty$ is introduced, which is called the arithmetic local parameter, and by using the arithmetic local parameter and the expression of the sigma function associated with the $(n,s)$ curve in terms of the tau function of the KP-hierarchy, it is proved that we have $\sigma(u)\in\mathbb{Z}[{\boldsymbol \l}_{\mathfrak{C}}]\langle\langle u \rangle\rangle$ and $\sigma(u)^2\in\mathbb{Z}[{\boldsymbol \l}]\langle\langle u\rangle\rangle$ for the $(n,s)$ curve (\cite[Theorem 2.3]{O-2018}). 
In \cite{Ouniversal}, in the case of the $(2,3)$ curve, the Hurwitz integrality of the elliptic sigma function is proved by an approach different from \cite{O-2018}. 
In \cite{Ouniversal,O-2018}, relationships of the Hurwitz integrality of the sigma functions with number theory are discussed.

We consider the telescopic curve associated with $A_m$. 
In \cite{Aya2}, an expression of the sigma function associated with the telescopic curve in terms of algebraic functions and integrals of an algebraic differential form is derived. 
Further, in \cite{Aya2}, an expression of the sigma function associated with the telescopic curve in terms of the tau function of the KP-hierarchy is also derived. 
In \cite{Aya2}, by using these expressions of the sigma function of the telescopic curve, it is proved that we have $\sigma(u)\in\mathbb{Q}[\boldsymbol{\lambda}]\langle\langle u\rangle\rangle$ for the telescopic curve. 
We assign degrees for $X_k$ and $\lambda_j^{(i)}$ as $\deg X_k=a_k$ and $\deg\lambda_j^{(i)}=j$. 
Let $\mathfrak{A}=\{\lambda_j^{(i)}\;|\;\mbox{$j$ is odd}\}$. 
In this paper, we generalize the arithmetic local parameter of the $(n,s)$ curve to the case of the telescopic curve (Section \ref{2022.4.3.111}). 
By using the arithmetic local parameter of the telescopic curve and the expression of the sigma function associated with the telescopic curve in terms of the tau function of the KP-hierarchy, 
we show $\sigma(u)\in\mathbb{Z}[{\boldsymbol \l}_{\mathfrak{A}}]\langle\langle u \rangle\rangle$ and $\sigma(u)^2\in\mathbb{Z}[\boldsymbol{\lambda}]\langle\langle u\rangle\rangle$ for the telescopic curve (Theorem \ref{maintheorem}). 
Let $I$ be the ideal of $\mathbb{C}[X_1,\dots,X_m]$ generated by $F_2,\dots,F_m$. 
For $f_1, f_2\in\mathbb{C}[X_1,\dots,X_m]$, if $f_1-f_2\in I$, we say that $f_1$ is congruent to $f_2$ modulo $I$. 
For a subring $R$ of $\mathbb{C}$, let $\mathcal{P}(R)$ be the set of $\sum \mathfrak{f}_{i_1,\dots,i_m}X_1^{i_1}\cdots X_m^{i_m}\in R[X_1,\dots,X_m]$ 
such that $\mathfrak{f}_{i_1,\dots,i_m}\in 2R$ or $(i_1,\dots,i_m)\in2{\mathbb Z}_{\geq 0}^m$.  
For $f\in\mathbb{Z}[{\boldsymbol \l}][X_1,\dots,X_m]$, 
let $\mathcal{S}(f)$ be the set of the subsets $\mathfrak{S}$ of ${\boldsymbol \l}$ such that 
$f\in\mathcal{P}(\mathbb{Z}[{\boldsymbol \l}_{\mathfrak{S}}])$. 
Let $G$ be the $(m-1)\times m$ matrix defined by 
\[
G=\left(\frac{\partial F_i}{\partial X_j}\right)_{2\leq i\leq m, \;1\le j\le m}
\]
and $G_k$ be the $(m-1)\times (m-1)$ matrix obtained by deleting the $k$-th column from $G$. 
Let $\langle A_m\rangle=\langle a_1,\dots, a_m\rangle$ and $g$ be the genus of the telescopic curve. 
We take an integer $k$ such that $1\le k\le m$ and $a_k$ is odd. 
We take a polynomial $f$ in $\mathbb{Z}[{\boldsymbol \l}][X_1,\dots,X_m]$ which is congruent to $\det G_k$ modulo $I$. 
In Lemma \ref{2023.10.22.2222}, we prove that if $\mathfrak{S}\in\mathcal{S}(f)$, we have $\sigma(u)\in\mathbb{Z}[{\boldsymbol \l}_{\mathfrak{S}}]\langle\langle u \rangle\rangle$.  
In Lemma \ref{2023.10.22.2222}, we have the following problems: 

\begin{itemize}
\item For any telescopic curve and any integer $k$ such that $1\le k\le m$ and $a_k$ is odd, is there a polynomial $f$ in $\mathbb{Z}[{\boldsymbol \l}][X_1,\dots,X_m]$ such that $f$ is congruent to $\det G_k$ modulo $I$ and $\mathcal{S}(f)\neq\emptyset$?

\item Find a polynomial $\mathscr{F}$ in $\mathbb{Z}[{\boldsymbol \l}][X_1,\dots,X_m]$ such that $\mathscr{F}$ gives the best result among all the polynomials in $\mathbb{Z}[{\boldsymbol \l}][X_1,\dots,X_m]$ which are congruent to $\det G_k$ modulo $I$. 

\end{itemize}

We solve these problems in Theorem \ref{2023.11.26.14598}, where for any integer $k$ such that $1\le k\le m$ and $a_k$ is odd we prove the following results: 

\vspace{1ex}

\noindent (i) The determinant $\det G_k$ is congruent to
\begin{equation}
\mathscr{F}=(-1)^{k+1}a_kX_1^{2h_1}\cdots X_m^{2h_m}+\sum \mathfrak{p}_{i_1,\dots,i_m}X_1^{i_1}\cdots X_m^{i_m}\label{2023.12.18.222}
\end{equation}
modulo $I$, where
\begin{itemize}

\item $(h_1,\dots,h_m)\in {\mathbb Z}_{\geq 0}^m$ such that $2\sum_{j=1}^ma_jh_j=2g-1+a_k$,

\item the summation in (\ref{2023.12.18.222}) is taken over $(i_1,\dots,i_m)\in {\mathbb Z}_{\geq 0}^m$ such that $\sum_{j=1}^ma_ji_j<2g-1+a_k$, 

\item $\mathfrak{p}_{i_1,\dots,i_m}\in\mathbb{Z}[{\boldsymbol \l}]$, 

\item $\mathscr{F}$ is homogeneous of degree $2g-1+a_k$ with respect to ${\boldsymbol \l}$ and $X_1,\dots,X_m$, 

\item if $(i_1,\dots,i_m)\neq(i_1',\dots,i_m')$ and $\mathfrak{p}_{i_1,\dots,i_m}\mathfrak{p}_{i_1',\dots,i_m'}\neq0$, then $\sum_{j=1}^ma_ji_j\neq\sum_{j=1}^ma_ji_j'$, 

\item if $(i_1,\dots,i_m)\notin2{\mathbb Z}_{\geq 0}^m$ and $\mathfrak{p}_{i_1,\dots,i_m}\neq0$, then $\sum_{j=1}^ma_ji_j\notin 2\langle A_m\rangle$. 
\end{itemize}

\noindent (ii) We have $\mathcal{S}(\mathscr{F})\neq\emptyset$. For any $\mathfrak{S}\in \mathcal{S}(\mathscr{F})$, we have $\sigma(u)\in\mathbb{Z}[{\boldsymbol \l}_{\mathfrak{S}}]\langle\langle u \rangle\rangle$. 

\vspace{1ex}

\noindent (iii) For any $f\in\mathbb{Z}[{\boldsymbol \l}][X_1,\dots,X_m]$ which is congruent to $\det G_k$ modulo $I$, we have $\mathcal{S}(f)\subseteq\mathcal{S}(\mathscr{F})$. 

\vspace{1ex}

We can apply Theorems \ref{maintheorem} and \ref{2023.11.26.14598} to any telescopic curve. 
For the $(n,s)$ curve, \cite[Theorem 2.3]{O-2018} gives the better result than Theorem \ref{maintheorem} (i) (Remark \ref{2023.12.19.1}). 
When we apply Theorem \ref{2023.11.26.14598} to the $(n,s)$ curve, we obtain the same result as \cite[Theorem 2.3]{O-2018} (Remark \ref{2023.12.19.2}). 
Therefore, Theorem \ref{2023.11.26.14598} includes \cite[Theorem 2.3]{O-2018}. 
In Examples \ref{2023.12.19.3}, \ref{2023.12.19.4}, and \ref{2023.12.19.5}, we consider the case of $m=3$ and $A_3=(4,6,5)$, $(4,6,7)$, and $(6,9,5)$, respectively. 
For these curves, Theorem \ref{2023.11.26.14598} gives the better result than Theorem \ref{maintheorem} (i).

In the case of hyperelliptic curves, more precise properties on the power series expansion of the sigma function are known. 
We consider the hyperelliptic curve of genus $g$ defined by 
\[X_2^2=X_1^{2g+1}+\lambda_4X_1^{2g-1}+\lambda_6X_1^{2g-2}+\cdots+\lambda_{4g}X_1+\lambda_{4g+2}.\]
By applying \cite[Theorem 2.3]{O-2018} to this curve, we obtain $\sigma(u)\in\mathbb{Z}[\{\lambda_{2i}\}_{i=2}^{2g+1}]\langle\langle u\rangle\rangle$. 
In \cite{Bunkova1,Ouniversal}, it is proved that we have $\sigma(u)\in\mathbb{Z}[2\l_4,8\l_6]\langle\langle u\rangle\rangle$ for $g=1$. 
In \cite{Bunkova1}, it is conjectured that we have $\sigma(u)\in\mathbb{Z}[2\l_4,24\l_6]\langle\langle u\rangle\rangle$ for $g=1$. 
In \cite[Corollary 2]{AB2020}, it is proved that we have $\sigma(u)\in\mathbb{Z}[\lambda_4,\lambda_6,\lambda_8,2\lambda_{10}]\langle\langle u\rangle\rangle$ for $g=2$.

\section{Telescopic sequences}

For an integer $m\geq 2$, let $A_m=(a_1,\dots, a_m)$ be a sequence of positive integers
such that $\mathrm{gcd}(a_1,\dots, a_m)=1$ and 
\[
\frac{a_i}{d_i}\in \left\langle\frac{a_1}{d_{i-1}}, \dots, \frac{a_{i-1}}{d_{i-1}}\right\rangle, 
\quad
2\leq i\leq m,
\]
where $d_i=\mathrm{gcd}(a_1,\dots, a_i)$. 
This sequence $A_m$ is called a \textit{telescopic sequence} (cf. \cite{KP}). 
Let $\langle A_m\rangle=\langle a_1,\dots, a_m\rangle$. 
The set $\mathbb{Z}_{\ge0}\backslash\langle A_m\rangle$ is a finite set. 
The number of elements of $\mathbb{Z}_{\ge0}\backslash\langle A_m\rangle$ is called the \textit{genus} of $\langle A_m\rangle$. 
Let $g$ be the genus of $\langle A_m\rangle$. 
Then the following relation holds: 
\begin{equation}
g=\frac{1}{2}\left\{1-a_1+\sum_{i=2}^m\left(\frac{d_{i-1}}{d_i}-1\right)a_i\right\}\label{2023.10.18.1}
\end{equation}
(cf. \cite{NW}). 
The following lemma is important in the proof of Theorem \ref{2023.11.26.14598}. 

\begin{lem}\label{2023.10.26.1}
Let $a$ be a positive odd integer such that $a\in\langle A_m \rangle$. 
Then we have $2g-1+a\in 2\langle A_m \rangle$. 
\end{lem}

\begin{proof}
Let $a=k_1a_1+\cdots+k_ma_m$, where $k_1,\dots,k_m\in\mathbb{Z}_{\ge0}$. 
We prove this lemma by induction on $m$. 
First, we prove this lemma for $m=2$. 
Note that $k_1a_1+k_2a_2$ is odd. 
If $(a_1,k_1)=(\mbox{odd}, \mbox{odd})$, then we have $(a_2,k_2)=(\mbox{odd}, \mbox{even})$, $(\mbox{even}, \mbox{odd})$, or $(\mbox{even}, \mbox{even})$. 
If $(a_1,k_1)=(\mbox{odd}, \mbox{even})$, $(\mbox{even}, \mbox{odd})$, or $(\mbox{even}, \mbox{even})$, then we have $(a_2,k_2)=(\mbox{odd}, \mbox{odd})$. 
In the case of $(a_1,k_1,a_2,k_2)=(\mbox{odd}, \mbox{odd}, \mbox{odd}, \mbox{even})$, $(\mbox{odd}, \mbox{odd}, \mbox{even}, \mbox{even})$, or $(\mbox{even}, \mbox{odd}, \mbox{odd}, \mbox{odd})$, 
we have 
\[2g-1+a=(k_1-1)a_1+(a_1-1+k_2)a_2\in 2\langle a_1,a_2 \rangle.\]
In the case of $(a_1,k_1,a_2,k_2)=(\mbox{odd}, \mbox{odd}, \mbox{even}, \mbox{odd})$, $(\mbox{odd}, \mbox{even}, \mbox{odd}, \mbox{odd})$, or $(\mbox{even}, \mbox{even}, \mbox{odd}, \mbox{odd})$, 
we have 
\[2g-1+a=(a_2+k_1-1)a_1+(k_2-1)a_2\in 2\langle a_1,a_2 \rangle.\] 
Therefore, we obtain the statement of the lemma for $m=2$. 
For an integer $n\ge3$, we assume that the statement of the lemma holds for $m=n-1$. 
We prove this lemma for $m=n$. 
For $1\le i\le n-1$, let $a_i'=a_i/d_{n-1}$ and $d_i'=d_i/d_{n-1}$. 
For $1\le i\le n-1$, we have $d_i'=\mathrm{gcd}(a_1',\dots, a_i')$. 
The sequence $(a_1',\dots,a_{n-1}')$ is also a telescopic sequence. 
Let $g'$ be the genus of $\langle a_1',\dots,a_{n-1}'\rangle$. 
Note that $k_1a_1+\cdots+k_na_n$ is odd. 
Let $a'=k_1a_1'+\cdots+k_{n-1}a_{n-1}'$. 
If $(d_{n-1},a')=(\mbox{odd}, \mbox{odd})$, then we have $(a_n,k_n)=(\mbox{odd}, \mbox{even})$, $(\mbox{even}, \mbox{odd})$, or $(\mbox{even}, \mbox{even})$. 
If $(d_{n-1},a')=(\mbox{odd}, \mbox{even})$, $(\mbox{even}, \mbox{odd})$, or $(\mbox{even}, \mbox{even})$, then we have $(a_n,k_n)=(\mbox{odd}, \mbox{odd})$. 
We consider the case of $(d_{n-1},a',a_n,k_n)=(\mbox{odd}, \mbox{odd}, \mbox{odd}, \mbox{even})$, $(\mbox{odd}, \mbox{odd}, \mbox{even}, \mbox{even})$, or $(\mbox{even}, \mbox{odd}, \mbox{odd}, \mbox{odd})$. 
We have 
\begin{align*}
2g-1+a&=d_{n-1}\left\{-a_1'+\sum_{i=2}^{n-1}\left(\frac{d_{i-1}'}{d_i'}-1\right)a_i'+a'\right\}+(d_{n-1}-1+k_n)a_n\\
&=d_{n-1}(2g'-1+a')+(d_{n-1}-1+k_n)a_n.
\end{align*}
By the induction hypothesis, we have $2g'-1+a'\in2\langle a_1',\dots,a_{n-1}'\rangle$. 
Thus, we have $d_{n-1}(2g'-1+a')\in2\langle a_1,\dots,a_{n-1}\rangle$. 
Since $d_{n-1}-1+k_n$ is even, we have $2g-1+a\in 2\langle a_1,\dots,a_n\rangle$. 
We consider the case of $(d_{n-1},a',a_n,k_n)=(\mbox{odd}, \mbox{odd}, \mbox{even}, \mbox{odd})$, $(\mbox{odd}, \mbox{even}, \mbox{odd}, \mbox{odd})$, or $(\mbox{even}, \mbox{even}, \mbox{odd}, \mbox{odd})$. 
We have 
\begin{align*}
2g-1+a&=d_{n-1}\left\{-a_1'+\sum_{i=2}^{n-1}\left(\frac{d_{i-1}'}{d_i'}-1\right)a_i'+a'+a_n\right\}+(k_n-1)a_n\\
&=d_{n-1}(2g'-1+a'+a_n)+(k_n-1)a_n.
\end{align*}
Since $(a_1,\dots,a_n)$ is a telescopic sequence, we have $a_n\in\langle a_1',\dots,a_{n-1}'\rangle$. 
By the induction hypothesis, we have $2g'-1+a'+a_n\in2\langle a_1',\dots,a_{n-1}'\rangle$. 
Thus, we have $d_{n-1}(2g'-1+a'+a_n)\in2\langle a_1,\dots,a_{n-1}\rangle$. 
Since $k_n-1$ is even, we have $2g-1+a\in 2\langle a_1,\dots,a_n\rangle$. 
Therefore, we obtain the statement of the lemma for $m=n$. 
By induction, we obtain the statement of the lemma for any $m\ge2$. 
\end{proof}

\section{Telescopic curves}\label{2023.11.26.752}

In this section, we briefly review the definition of 
telescopic curves following \cite{Miu,Aya1,Aya2}.

Let $A_m=(a_1,\dots, a_m)$ be a telescopic sequence such that $a_i\ge2$ for any $i$. 
Let
\[
B(A_m)=\left\{(\ell_1,\dots, \ell_m)\in {\mathbb Z}_{\geq 0}^m\,\middle|\,0\leq \ell_i\leq \frac{d_{i-1}}{d_i}-1\,\, \text{ for }\,\,2\leq i\leq m\right\}.
\]

\begin{lem}[\cite{BertinCarbonne,KP}]\label{lem-2-1}
For any $a\in \langle A_m \rangle$, we have the unique element $(k_1,\dots,k_m)$ of $B(A_m)$ such that
\[
\sum_{i=1}^ma_ik_i=a.
\]
\end{lem}

By this lemma, for any $2\leq i\leq m$, we have the unique element 
$(\ell_{i,1},\dots, \ell_{i,m})\in B(A_m)$ satisfying 
\begin{equation}
\sum_{j=1}^m a_j \ell_{i,j}= a_i\frac{d_{i-1}}{d_i}.\label{homo}
\end{equation}

\begin{lem}[\cite{Aya2}]\label{defining}
For any $2\le i\le m$, we have $\ell_{i,j}=0$ for $j\geq i$. 
\end{lem}

Consider the $m-1$ polynomials in $m$ variables $X=(X_1,\dots, X_m)$ given by
\begin{equation}
F_i(X)=X_i^{d_{i-1}/d_i}-\prod_{j=1}^{i-1} X_j^{\ell_{i,j}}-\sum \lambda_{a_id_{i-1}/d_i-\sum_{k=1}^m a_kj_k}^{(i)}X_1^{j_1}\cdots X_m^{j_m},\label{eq-2-5}
\end{equation}
where $2\leq i\leq m$, $\lambda_{a_id_{i-1}/d_i-\sum_{k=1}^m a_kj_k}^{(i)}\in\mathbb{C}$, 
and the summation is taken over 
$(j_1,\dots, j_m)\in B(A_m)$ such that 
\[
\sum_{k=1}^m a_kj_k<a_i\frac{d_{i-1}}{d_i}.
\]
Let $V^{\mathrm{aff}}$ be the common zeros of $F_2,\dots, F_m$:
\[
V^{\mathrm{aff}}=\{(X_1,\dots, X_m)\in\mathbb{C}^m\,|\,F_i(X_1,\dots, X_m)=0,\, 2\leq i\leq m\}.
\]
In \cite{Miu,Aya1}, $V^{\mathrm{aff}}$ is proved to be an affine algebraic curve.
We assume that $V^{\mathrm{aff}}$ is non-singular. Let $V$ be the compact Riemann surface corresponding to $V^{\mathrm{aff}}$. Then $V$ is obtained from $V^{\mathrm{aff}}$
by adding one point, say $\infty$ \cite{Miu,Aya1}. The genus of $V$ coincides with the genus $g$ of  $\langle A_m\rangle$, which is given by (\ref{2023.10.18.1}).  
We call $V$ the \textit{telescopic curve} 
associated with $A_m$. 
For $m=2$, for simplicity, we denote $\lambda_j^{(2)}$ by $\lambda_j$. 

\vspace{1ex}

\begin{ex}\label{2022.5.29.2}
(i) Let $n$ and $s$ be integers such that $2\le n<s$ and $\mathrm{gcd}(n,s)=1$. 
We consider the case of $m=2$ and $A_2=(n,s)$. 
If $(j_1,j_2)\in{\mathbb Z}_{\geq 0}^2$ satisfies $nj_1+sj_2<ns$, then we have $j_2<n$. Thus, we have $(j_1,j_2)\in B(A_2)$. 
The polynomial $F_2$ is given by 
\begin{equation}
F_2(X)=X_2^n-X_1^s-\sum\lambda_{ns-nj_1-sj_2}X_1^{j_1}X_2^{j_2},\label{2023.12.9.1}
\end{equation}
where the summation is taken over $(j_1,j_2)\in{\mathbb Z}_{\geq 0}^2$ such that $nj_1+sj_2<ns$. 
The telescopic curve associated with $A_2=(n,s)$ is the \textit{$(n,s)$ curve} introduced in \cite{BEL-99-R}.

\vspace{1ex}

\noindent
(ii) For $m=3$ and $A_3=(4,6,5)$, 
the polynomials $F_2$ and $F_3$ are given by
\begin{align*}
F_2(X)&=X_2^2-X_1^3-\lambda_1^{(2)}X_2X_3-\lambda_2^{(2)}X_1X_2-\lambda_3^{(2)}X_1X_3-\lambda_4^{(2)}X_1^2-\lambda_6^{(2)}X_2-\lambda_7^{(2)}X_3\\
&-\lambda_8^{(2)}X_1-\lambda_{12}^{(2)}, \\
F_3(X)&=X_3^2-X_1X_2-\lambda_1^{(3)}X_1X_3-\lambda_2^{(3)}X_1^2-\lambda_4^{(3)}X_2-\lambda_5^{(3)}X_3-\lambda_6^{(3)}X_1-\lambda_{10}^{(3)}.
\end{align*}

\vspace{1ex}

\noindent
(iii) For $m=3$ and $A_3=(4,6,7)$, 
the polynomials $F_2$ and $F_3$ are given by
\begin{align*}
F_2(X)&=X_2^2-X_1^3-\lambda_1^{(2)}X_1X_3-\lambda_2^{(2)}X_1X_2-\lambda_4^{(2)}X_1^2-\lambda_5^{(2)}X_3-\lambda_6^{(2)}X_2-\lambda_8^{(2)}X_1-\lambda_{12}^{(2)},\\
F_3(X)&=X_3^2-X_1^2X_2-\lambda_1^{(3)}X_2X_3-\lambda_2^{(3)}X_1^3-\lambda_3^{(3)}X_1X_3-\lambda_4^{(3)}X_1X_2-\lambda_6^{(3)}X_1^2-\lambda_7^{(3)}X_3\\
&-\lambda_8^{(3)}X_2-\lambda_{10}^{(3)}X_1-\lambda_{14}^{(3)}.
\end{align*}

\vspace{1ex}

\noindent
(iv) For $m=3$ and $A_3=(6,9,5)$, 
the polynomials $F_2$ and $F_3$ are given by
\begin{align*}
F_2(X)&=X_2^2-X_1^3-\lambda_1^{(2)}X_1^2X_3-\lambda_2^{(2)}X_1X_3^2-\lambda_3^{(2)}X_1X_2-\lambda_4^{(2)}X_2X_3-\lambda_6^{(2)}X_1^2\\
&-\lambda_7^{(2)}X_1X_3-\lambda_8^{(2)}X_3^2-\lambda_9^{(2)}X_2-\lambda_{12}^{(2)}X_1-\lambda_{13}^{(2)}X_3-\lambda_{18}^{(2)},\\
F_3(X)&=X_3^3-X_1X_2-\lambda_1^{(3)}X_2X_3-\lambda_3^{(3)}X_1^2-\lambda_4^{(3)}X_1X_3-\lambda_5^{(3)}X_3^2-\lambda_6^{(3)}X_2-\lambda_9^{(3)}X_1\\
&-\lambda_{10}^{(3)}X_3-\lambda_{15}^{(3)}.
\end{align*}

\vspace{1ex}

\noindent(v)\footnote{This example is given in \cite{Miu,Aya2}.} Let $a$ and $b$ be integers such that $a>b\ge2$ and $\mathrm{gcd}(a,b)=1$. 
For $A_m=(a_1,\dots ,a_m)$, where $a_i=a^{m-i}b^{i-1}$, the polynomials $F_i$, $2\le i\le m$, are given by 
\[
F_i(X)=X_i^a-X_{i-1}^b-
\sum
\lambda_{aa_i-\sum_{k=1}^m a_kj_k}^{(i)}X_1^{j_1}\cdots X_m^{j_m},
\]
where the summation is taken over $(j_1,\dots, j_m)\in B(A_m)$ such that $\sum_{k=1}^m a_kj_k<aa_i$. 
\end{ex}

Let $I$ be the ideal of $\mathbb{C}[X_1,\dots,X_m]$ generated by $F_2,\dots,F_m$. 
Then $I$ is a prime ideal (cf. \cite{Miu,Aya1}). 
The coordinate ring $\mathbb{C}[X_1,\dots,X_m]/I$ of $V$ can be identified with the set of the meromorphic functions on $V$ which are holomorphic at any point except $\infty$ (cf. \cite{Miu,Aya1}). 
For $1\le k\le m$, let $x_k$ be the image of $X_k$ for the projection $\mathbb{C}[X_1,\dots,X_m]\to\mathbb{C}[X_1,\dots,X_m]/I$.  
In \cite{Miu,Aya1}, it is proved that $x_k$ has a pole of order $a_k$ at $\infty$. 
We assign degrees for $X_k$, $x_k$, and $\lambda_j^{(i)}$ as 
\[\deg X_k=\deg x_k=a_k,\qquad \deg\lambda_j^{(i)}=j.\] 
We denote by ${\boldsymbol \l}$ the set of all $\lambda_j^{(i)}$. 
For $2\le i\le m$, the polynomial $F_i(X)$ is homogeneous of degree $a_id_{i-1}/d_i$ with respect to the coefficients ${\boldsymbol \l}$ and the variables $X_1,\dots, X_m$. 
For $(k_1,\dots, k_m)\in\mathbb{Z}_{\ge0}^m$, we have $\deg x_1^{k_1}\cdots x_m^{k_m}=a_1k_1+\cdots+a_mk_m$. 
We define an ordering for the monomials $x_1^{k_1}\cdots x_m^{k_m},\;(k_1,\dots, k_m)\in B(A_m)$, according as the degree and denote them by 
$\varphi_i$, $i\geq 1$. In particular, we have $\varphi_1=1$. 
The set $\{\varphi_i\}_{i=1}^{\infty}$ is a basis of the coordinate ring $\mathbb{C}[X_1,\dots,X_m]/I$ over $\mathbb{C}$ (cf. \cite{Miu,Aya1,Aya2}). 
Further, we have the following more precise property. 

\begin{lem}\label{2022.6.28.222}
For any $(k_1,\dots,k_m)\in\mathbb{Z}_{\ge0}^m$, the monomial $x_1^{k_1}\cdots x_m^{k_m}$ can be uniquely expressed by the linear combination of $\varphi_i$ as follows: 
\begin{equation}
x_1^{k_1}\cdots x_m^{k_m}=\varphi_n+\sum_{i=1}^{n-1}\rho_i\varphi_i,\label{2022.3.21.1}
\end{equation}
where $\deg \varphi_{n}=\sum_{i=1}^ma_ik_i$, $\rho_i\in\mathbb{Z}[{\boldsymbol \l}]$ for $1\le i\le n-1$, and the right hand side of (\ref{2022.3.21.1}) is homogeneous of degree $\sum_{i=1}^ma_ik_i$ with respect to ${\boldsymbol \l}$ and $x_1,\dots,x_m$. 
\end{lem}

\begin{proof}
The proof of this lemma is similar to that of the lemma in \cite[p.~1412]{Miu}. 
For the sake to be complete and self-contained, we give a proof of this lemma. 
For $(k_1,\dots,k_m)$, \\ $(k_1',\dots,k_m')\in\mathbb{Z}_{\ge0}^m$, we define $(k_1,\dots,k_m)<(k_1',\dots,k_m')$ if and only if 
\begin{itemize}
\item $\sum_{i=1}^ma_ik_i<\sum_{i=1}^ma_ik_i'$ or 

\item $\sum_{i=1}^ma_ik_i=\sum_{i=1}^ma_ik_i'$ and $k_1=k_1', \;k_2=k_2', \dots, k_{i-1}=k_{i-1}', \;k_i>k_i'$ for some $1\le i\le m$. 
\end{itemize}
For $(k_1,\dots,k_m)=(0,\dots,0)$, it is trivial that $x_1^{k_1}\cdots x_m^{k_m}$ can be expressed in the form of (\ref{2022.3.21.1}). 
We take $(\ell_1,\dots,\ell_m)\in\mathbb{Z}_{\ge0}^m$ such that $(\ell_1,\dots,\ell_m)\neq(0,\dots,0)$. 
We assume that $x_1^{k_1}\cdots x_m^{k_m}$ can be expressed in the form of (\ref{2022.3.21.1}) for any $(k_1,\dots,k_m)\in\mathbb{Z}_{\ge0}^m$ such that $(k_1,\dots,k_m)<(\ell_1,\dots,\ell_m)$. 
If $(\ell_1,\dots,\ell_m)\in B(A_m)$, then it is trivial that $x_1^{\ell_1}\cdots x_m^{\ell_m}$ can be expressed in the form of (\ref{2022.3.21.1}). 
We assume $(\ell_1,\dots,\ell_m)\notin B(A_m)$. 
Then there exists an integer $i$ such that $2\le i\le m$ and $\ell_i\ge d_{i-1}/d_i$. 
From (\ref{eq-2-5}), we have 
\begin{align*}
&x_1^{\ell_1}\cdots x_m^{\ell_m}\\
&=x_1^{\ell_1}\cdots x_{i-1}^{\ell_{i-1}}x_i^{\ell_i-d_{i-1}/d_i}\left(\prod_{j=1}^{i-1} x_j^{\ell_{i,j}}+
\sum \lambda_{a_id_{i-1}/d_i-\sum_{k=1}^m a_kj_k}^{(i)}x_1^{j_1}\cdots x_m^{j_m}\right)x_{i+1}^{\ell_{i+1}}\cdots x_m^{\ell_m}.
\end{align*}
The right hand side of the above equation is homogeneous of degree $\sum_{i=1}^ma_i\ell_i$ with respect to ${\boldsymbol \l}$ and $x_1,\dots,x_m$. 
For any $x_1^{k_1}\cdots x_m^{k_m}$ in the right hand side of the above equation, we have $(k_1,\dots,k_m)<(\ell_1,\dots,\ell_m)$. 
Thus, $x_1^{\ell_1}\cdots x_m^{\ell_m}$ can be expressed in the form of (\ref{2022.3.21.1}). 
By induction, $x_1^{k_1}\cdots x_m^{k_m}$ can be expressed in the form of (\ref{2022.3.21.1}) for any $(k_1,\dots,k_m)\in\mathbb{Z}_{\ge0}^m$. 
Since the set $\{\varphi_i\}_{i=1}^{\infty}$ is a basis of the coordinate ring of the telescopic curve $V$ over $\mathbb{C}$, the expression of (\ref{2022.3.21.1}) is unique. 
\end{proof}

In the proof of Lemma \ref{2022.6.28.222}, a method to express $x_1^{k_1}\cdots x_m^{k_m}$ as the linear combination of $\varphi_i$ explicitly is given. 
Let $(w_1,\dots, w_g)$ be the gap sequence at $\infty$:
\[
\{w_i\,|\,1\leq i\leq g\}=\mathbb Z_{\geq0}
\backslash\langle A_m \rangle,
\quad
w_1<\cdots<w_g. 
\] 
In particular, we have $w_1=1$. 
The numbers $a_1,\dots, a_m$ generate the semigroup of 
non-gaps at $\infty$. 
Let $G$ be the $(m-1)\times m$ matrix defined by 
\[
G=\left(\frac{\partial F_i}{\partial X_j}\right)_{2\leq i\leq m, \;1\le j\le m}
\]
and $G_k$ be the $(m-1)\times (m-1)$ matrix obtained by deleting the $k$-th column from $G$. 
Let $P=(x_1,\dots,x_m)\in V$. 
A basis of the vector space consisting of holomorphic 1-forms on $V$ is given by
\[
\omega_i=-\frac{\varphi_{g+1-i}}{\det G_1(P)}dx_1,\quad 1\le i\le g, 
\]
where $\det G_1$ is the determinant of $G_1$ (cf. \cite{Aya1}). 

\begin{lem}[\cite{NW,KP}]\label{lem-2-2}
We have $w_{g}=2g-1$. 
\end{lem}

\begin{lem}[\cite{Aya1}]\label{2023.12.23.1}
The holomorphic $1$-form $\omega_g$ has a zero of order $2g-2$ at $\infty$. 
\end{lem}

From Lemmas \ref{lem-2-2} and \ref{2023.12.23.1}, we find that the Riemann constant for the telescopic curve with the base point $\infty$ is a half-period. 



\section{Klein's fundamental $2$-form}

A Klein's fundamental $2$-form plays important roles in the theory of the sigma functions. 
We recall its definition. 

We consider the telescopic curve $V$ of genus $g$ associated with $A_m=(a_1, \dots, a_m)$. 
Let $K_V$ be the canonical bundle of $V$. 
For $i=1,2$, let $\pi_i : V\times V\to V$ be the projection to the $i$-th component. 
A section of $\pi_1^*K_V\otimes \pi_2^*K_V$ is called a \textit{bilinear form} on $V\times V$ (cf. \cite{Fay-1973, N1, MSDCS}). 

\begin{defn}
A meromorphic bilinear form $\omega(P,Q)$ on $V\times V$ is called a \textit{Klein's fundamental $2$-form} if 
the following conditions are satisfied. 

\vspace{1ex}

\noindent(i) $\omega(Q,P)=\omega(P,Q)$ for any $P, Q\in V$. 

\vspace{1ex}

\noindent (ii) $\omega(P,Q)$ is holomorphic at any point except $\{(R,R)\;|\;R\in V\}$. 

\vspace{1ex}

\noindent (iii) For any $R\in V$, take a local parameter $t$ around $R$. 
Then $\omega(P,Q)$ has the following form around $(R,R)$: 
\[
\omega(P,Q)=\left(\frac{1}{(t_P-t_Q)^2}+f\left(t_P, t_Q\right)\right)dt_Pdt_Q, 
\]
where $t_P=t(P)$, $t_Q=t(Q)$, and $f\left(t_P, t_Q\right)$ is a holomorphic function of $t_P$ and $t_Q$. 
\end{defn}

For a Klein's fundamental $2$-form $\omega(P,Q)$ and complex numbers $\{c_{i,j}\}_{i,j=1}^g$ such that $c_{i,j}=c_{j,i}$, 
\[
\omega(P,Q)+\sum_{i,j=1}^gc_{i,j}\omega_i(P)\omega_j(Q)
\]
is also a Klein's fundamental $2$-form. 

For the telescopic curve $V$, a Klein's fundamental $2$-form is algebraically constructed in \cite{Aya1}. 
We recall its construction. 
Note that the construction inherits all steps of the classical construction in \cite{Baker} that was recently recapitulated and generalized in \cite{EEL,N1} for $(n,s)$ curves. 
We define the meromorphic bilinear form $\widehat{\omega}(P,Q)$ on $V\times V$ by
\[
\widehat{\omega}(P,Q)=d_Q\Omega(P,Q)+
\sum_{i=1}^g \omega_i(P)\eta_i(Q),
\]
where $P=(x_1,\dots,x_m)$ and $Q=(y_1,\dots,y_m)$ are points on $V$,
\[
\Omega(P,Q)=\frac{\det H(P,Q)}{(x_1-y_1)\det G_1(P)}dx_1,
\]
$H=(h_{i,j})_{2\leq i,j\leq m}$ with
\[
h_{i,j}=\frac{F_i(y_1, \dots ,y_{j-1} ,x_j, x_{j+1}, \dots, x_m)-F_i(y_1, \dots, y_{j-1}, y_j, x_{j+1}, \dots, x_m)}{x_j-y_j},
\]
and $\eta_i$ is a meromorphic 1-form on $V$ which is holomorphic at any point except $\infty$. 
Here, $d_Q\Omega(P,Q)$ means the derivative of $\Omega(P,Q)$ with respect to $Q$. 

\begin{lem}[{\cite[Lemma 4.7]{Aya1}, \cite[Lemma 6]{N1}}]
The set 
\[
\left\{\frac{\varphi_i}{\det G_1(P)}dx_1\right\}_{i=1}^{\infty}
\]
is a basis of the vector space consisting of the meromorphic $1$-forms on $V$ which are holomorphic at any point except $\infty$. 
\end{lem}

Let 
\[
\sum_{i=1}^g\omega_i(P)\eta_i(Q)=\frac{\sum c_{i_1,\dots,i_m;j_1,\dots,j_m}x_1^{i_1}\cdots x_m^{i_m}y_1^{j_1}\cdots y_m^{j_m}}{\det G_1(P)\det G_1(Q)}dx_1dy_1, 
\]
where $(i_1,\dots,i_m), (j_1,\dots,j_m)\in B(A_m)$ and $c_{i_1,\dots,i_m;j_1,\dots,j_m}\in\mathbb{C}$. 

\begin{lem}[{\cite[Theorem 4.1 (i)]{Aya1}, \cite[Proposition 2 (ii)]{N1}}]\label{2022.8.6.1234} 
It is possible to take $\{\eta_i\}_{i=1}^g$ such that $\widehat{\omega}(Q,P)=\widehat{\omega}(P,Q)$, $c_{i_1,\dots,i_m;j_1,\dots,j_m}\in\mathbb{Q}[{\boldsymbol \l}]$, and 
$c_{i_1,\dots,i_m;j_1,\dots,j_m}$ is homogeneous of degree $2(2g-1)-\sum_{k=1}^ma_k(i_k+j_k)$ with respect to ${\boldsymbol \l}$ if $c_{i_1,\dots,i_m;j_1,\dots,j_m}\neq0$. 
\end{lem}


\begin{lem}[{\cite[Theorem 4.1 (ii)]{Aya1}, \cite[Proposition 2 (i)]{N1}}]
If we take $\{\eta_i\}_{i=1}^g$ as in Lemma \ref{2022.8.6.1234}, then $\widehat{\omega}(P,Q)$ becomes a Klein's fundamental $2$-form. 
\end{lem}

\section{Sigma functions of telescopic curves}\label{2022.3.19.9}

We take $\{\eta_i\}_{i=1}^g$ as in Lemma \ref{2022.8.6.1234}. 
We take a canonical basis $\{\mathfrak{a}_i, \mathfrak{b}_i\}_{i=1}^g$ in the one-dimensional homology group of the telescopic curve $V$ 
and define the period matrices by
\[
2\omega'=\left(\int_{\mathfrak{a}_j}\omega_i\right),\quad
2\omega''=\left(\int_{\mathfrak{b}_j}\omega_i\right),\quad -2\eta'=\left(\int_{\mathfrak{a}_j}\eta_i\right),
\quad
-2\eta''=\left(\int_{\mathfrak{b}_j}\eta_i\right).
\]
The normalized period matrix is given by $\tau=(\omega')^{-1}\omega''$. 
Let $\delta=\tau\delta'+\delta''$ with $\delta',\delta''\in {\mathbb R}^g$ 
be the Riemann's constant with respect to $(\{\mathfrak{a}_i,\mathfrak{b}_i\}_{i=1}^g,\infty)$. 
We denote the imaginary unit by $\textbf{i}$.
The sigma function $\sigma(u)$ associated with the curve $V$, $u={}^t(u_1, \dots, u_g)$, is defined by
\begin{equation}
\sigma(u)=C\exp\left(\frac{1}{2}{}^tu\eta'(\omega')^{-1}u\right)
\theta\begin{bmatrix}\delta'\\ \delta'' \end{bmatrix}\left((2\omega')^{-1}u,\tau\right),
\label{2023.8.23.1357}
\end{equation}
where $\theta\begin{bmatrix}\delta'\\ \delta'' \end{bmatrix}(u,\tau)$ is the Riemann's theta function with the characteristics $\begin{bmatrix}\delta'\\ \delta'' \end{bmatrix}$ defined by 
\[
\theta\begin{bmatrix}\delta'\\ \delta'' \end{bmatrix}(u,\tau)=\sum_{n\in\mathbb{Z}^g}\exp\{\pi\textbf{i}\;{}^t(n+\delta')\tau(n+\delta')+2\pi\textbf{i}\;{}^t(n+\delta')(u+\delta'')\},
\]
and $C$ is a non-zero constant which is fixed in Theorem \ref{2023.8.23.1}. 
Since $\delta$ is a half-period, $\sigma(u)$ vanishes on the Abel-Jacobi image of the $(g-1)$-st symmetric product of the telescopic curve. 
We have the following proposition. 

\begin{prop}[\cite{Aya1, N1}] \label{period}
For $m_1,m_2\in\mathbb{Z}^g$ and $u\in\mathbb{C}^g$, we have 
\[
\frac{\sigma(u+2\omega'm_1+2\omega''m_2)}{\sigma(u)}
=(-1)^{2({}^t\delta'm_1-{}^t\delta''m_2)+{}^tm_1m_2}\exp\{{}^t(2\eta'm_1+2\eta''m_2)(u+\omega'm_1+\omega''m_2)\}.
\]
\end{prop}

A sequence of non-negative integers $\mu=(\mu_1,\mu_2,\dots,\mu_{\ell})$ such that $\mu_1\ge\mu_2\ge\cdots\ge\mu_{\ell}$ is called a {\it partition}. 
For a partition $\mu=(\mu_1,\mu_2,\dots,\mu_{\ell})$, let $|\mu|=\mu_1+\mu_2+\cdots+\mu_{\ell}$. 
For $n\ge0$, let $p_n(T)$ be the polynomial of $T_1, T_2, \dots$ defined by 
\begin{equation}
\sum_{i=0}^{\infty}\frac{1}{i!}\left(\sum_{j=1}^{\infty}T_jk^j\right)^i=\sum_{n=0}^{\infty}p_n(T)k^n,\label{2022.3.16.1}
\end{equation}
where $k$ is a variable, i.e., $p_n(T)$ is the coefficient of $k^n$ in the left hand side of (\ref{2022.3.16.1}). 
For example, we have 
\[p_0(T)=1,\;\;\;\;p_1(T)=T_1,\;\;\;\;p_2(T)=T_2+\frac{T_1^2}{2},\;\;\;\;p_3(T)=T_3+T_1T_2+\frac{T_1^3}{6}.\]
For $n<0$, let $p_n(T)=0$. 

\begin{lem}\label{2022.3.19.99}
For $n\ge1$, we have 
\[p_n(T)=\sum\frac{T_1^{k_1}\cdots T_n^{k_n}}{k_1!\cdots k_n!},\]
where the summation is taken over $(k_1,\dots,k_n)\in\mathbb{Z}_{\ge0}^n$ satisfying 
\[\sum_{j=1}^njk_j=n.\]
\end{lem}

\begin{proof}
By comparing the coefficients of $k^n$ in (\ref{2022.3.16.1}), we obtain the statement of the lemma. 
\end{proof}

For an arbitrary partition $\mu=(\mu_1,\mu_2,\dots,\mu_{\ell})$, the Schur function $S_{\mu}(T)$ is defined by
\[S_{\mu}(T)=\det \left(p_{\mu_i-i+j}(T)\right)_{1\le i,j\le \ell}.\]
For the telescopic curve $V$ associated with $A_m=(a_1,\dots,a_m)$, we define the partition by 
\[
\mu(A_m)=(w_g,\dots,w_1)-(g-1,\dots,0).
\]

\vspace{1ex}


\begin{lem}
The Schur function $S_{\mu(A_m)}(T)$ is a polynomial of the variables $T_{w_1},\dots, T_{w_g}$. 
\end{lem}

\begin{proof}
We can prove this lemma as in the case of $(n,s)$ curves (cf. \cite[Section 4]{BEL-99-R}). 
\end{proof}

\vspace{1ex}

\begin{thm}[{\cite[Theorem 7]{Aya2}}]\label{2023.8.23.1}
The sigma function $\sigma(u)$ is a holomorphic function on $\mathbb{C}^g$ and we have the unique constant $C$ in (\ref{2023.8.23.1357}) such that the series expansion of $\sigma(u)$ around the origin has the following form: 
\begin{equation}
\sigma(u)=S_{\mu(A_m)}(T)|_{T_{w_i}=u_i}+\sum_{w_1n_1+\cdots+w_gn_g>|\mu(A_m)|}\varepsilon_{n_1,\dots,n_g}\frac{u_1^{n_1}\cdots u_g^{n_g}}{n_1!\cdots n_g!},\label{4.27.1}
\end{equation}
where $\varepsilon_{n_1,\dots,n_g}\in\mathbb{Q}[{\boldsymbol \l}]$ and $\varepsilon_{n_1,\dots,n_g}$ is homogeneous
of degree $w_1n_1+\cdots+w_gn_g-|\mu(A_m)|$ with respect to ${\boldsymbol \l}$ if $\varepsilon_{n_1,\dots,n_g}\neq0$.
\end{thm}

We take the constant $C$ in (\ref{2023.8.23.1357}) such that the expansion (\ref{4.27.1}) holds. 
Then the sigma function $\sigma(u)$ does not depend on the choice of a canonical basis $\{\mathfrak{a}_i,\mathfrak{b}_i\}_{i=1}^g$ in the one-dimensional homology group of the curve $V$ and is determined only 
by the coefficients ${\boldsymbol \l}$ of the defining equations of the curve $V$.

\section{Arithmetic local parameter for the telescopic curve}\label{2022.4.3.111}

Since $\mathrm{gcd}(a_1, \dots, a_m)=1$, we can take $(b_1,\dots,b_m)\in\mathbb{Z}^m$ such that 
\begin{equation}
a_1b_1+\cdots+a_mb_m=-1.\label{2023.8.24.1}
\end{equation}
Note that $(b_1,\dots,b_m)\in\mathbb{Z}^m$ satisfying (\ref{2023.8.24.1}) is not uniquely determined. 
We fix $(b_1,\dots,b_m)\in\mathbb{Z}^m$ satisfying (\ref{2023.8.24.1}). 
We consider the defining equations (\ref{eq-2-5}) of the telescopic curve $V$.  
We consider the matrix
\begin{equation}
D=\left(\begin{matrix}-\ell_{2,1} & d_1/d_2 & 0 & \cdots & 0 \\
-\ell_{3,1} & -\ell_{3,2} & d_2/d_3 & \cdots & 0  \\
\vdots & \vdots & \vdots & \ddots & \vdots\\
-\ell_{m,1} & -\ell_{m,2}  &\cdots & -\ell_{m, m-1} & d_{m-1}/d_m \\
b_1 & b_2  & \cdots & b_{m-1} & b_m
\end{matrix}\right)\in M_m(\mathbb{Z}).\label{2024.1.4.1}
\end{equation}

\begin{lem}\label{lem11}
We have $\det(D)=(-1)^m$. 
\end{lem}

\begin{proof}
From (\ref{homo}) and Lemma \ref{defining}, we have 
\begin{equation}
D\left(\begin{matrix}a_1\\ \vdots \\ a_m\end{matrix}\right)=\left(\begin{matrix}0\\ \vdots \\ 0 \\ -1\end{matrix}\right).\label{det}
\end{equation}

By multiplying some elementary matrices whose determinants are $1$ on the left, the equation (\ref{det}) becomes 
\[\widehat{D}\left(\begin{matrix}a_1\\ \vdots \\ a_m\end{matrix}\right)=\left(\begin{matrix}0\\ \vdots \\ 0 \\ -1\end{matrix}\right),\]
where 
\[
\widehat{D}=\left(\begin{matrix}
e_2 & d_1/d_2 & 0 & \cdots & 0\\
e_3 & 0 & d_2/d_3 & \cdots & 0\\
\vdots & \vdots & \vdots & \ddots & \vdots \\
e_m & 0 & 0 & \cdots & d_{m-1}/d_m \\ 
e_1 & 0 & 0 & \cdots & 0
\end{matrix}\right)
\]
with some $e_1,\dots,e_m\in\mathbb{Q}$. 
From the above equation, we obtain $e_1=-1/a_1$. 
We have 
\[\det(D)=\det(\widehat{D})=(-1)^{m-1}\frac{d_1}{d_2}\cdot\frac{d_2}{d_3}\cdots\frac{d_{m-1}}{d_m}\cdot e_1=(-1)^m.\]

\end{proof}

Let 
\begin{equation}
t=x_1^{b_1}\cdots x_m^{b_m}.\label{t}
\end{equation}
Since $t$ has a zero of order $1$ at $\infty$, we can regard $t$ as a local parameter of $V$ around $\infty$. 
We call $t$ an {\it arithmetic local parameter} as in the case of \cite{O-2018}. 
For $1\le i\le m$, we consider the expansion of $x_i$ around $\infty$ with respect to $t$ 
\begin{equation}
x_i=\frac{1}{t^{a_i}}\sum_{k=0}^{\infty}p_{i,k}t^k,\;\;\;\;\;p_{i,k}\in\mathbb{C}.\label{xi}
\end{equation}
From (\ref{eq-2-5}) and (\ref{xi}), for $2\le i\le m$, we obtain
\begin{align}
&\left(\sum_{k=0}^{\infty}p_{i,k}t^k\right)^{d_{i-1}/d_i}=\prod_{j=1}^{i-1}\left(\sum_{k=0}^{\infty}p_{j,k}t^k\right)^{\ell_{i,j}} \non \\ 
&+\sum \lambda_{a_id_{i-1}/d_i-\sum_{k=1}^m a_kj_k}^{(i)}t^{a_id_{i-1}/d_i-\sum_{k=1}^ma_kj_k}\left(\sum_{k=0}^{\infty}p_{1,k}t^k\right)^{j_1}\cdots \left(\sum_{k=0}^{\infty}p_{m,k}t^k\right)^{j_m},\label{ar}
\end{align}
where the summation is taken over 
$(j_1, \dots, j_m)\in B(A_m)$ such that 
\[
\sum_{k=1}^m a_kj_k<a_i\frac{d_{i-1}}{d_i}.
\]

\begin{prop}\label{prop11}
We have $p_{1,0}=p_{2,0}=\cdots=p_{m,0}=1$. 
\end{prop}

\begin{proof}
By comparing the coefficients of $t^0$ in (\ref{ar}), we obtain 
\begin{equation}
p_{i,0}^{d_{i-1}/d_i}=\prod_{j=1}^{i-1}p_{j,0}^{\ell_{i,j}}\label{p1}
\end{equation}
for $2\le i\le m$. 
By substituting (\ref{xi}) into (\ref{t}), we obtain 
\[1=\left(\sum_{k=0}^{\infty}p_{1,k}t^k\right)^{b_1}\cdots \left(\sum_{k=0}^{\infty}p_{m,k}t^k\right)^{b_m}.\]
We divide the set $\{1,2,\dots,m\}$ into the two sets $\{\alpha_1,\dots,\alpha_s\}$ and $\{\alpha_{s+1},\dots,\alpha_m\}$, where 
$b_{\alpha_1},\dots,b_{\alpha_s}$ are negative integers and $b_{\alpha_{s+1}},\dots,b_{\alpha_m}$ are non-negative integers. 
Then we have 
\begin{equation}
\prod_{j=1}^s\left(\sum_{k=0}^{\infty}p_{\alpha_j,k}t^k\right)^{-b_{\alpha_j}}=\prod_{j=s+1}^m\left(\sum_{k=0}^{\infty}p_{\alpha_j,k}t^k\right)^{b_{\alpha_j}}.\label{kuro}
\end{equation}
By comparing the coefficients of $t^0$ in (\ref{kuro}), we obtain 
\begin{equation}
p_{1,0}^{b_1}\cdots p_{m,0}^{b_m}=1.\label{p2}
\end{equation}
We perform the following elementary row operations on the matrix $D$ defined in (\ref{2024.1.4.1}). 
First, for $1\le i\le m$, we multiply the $i$-th row by $a_1/d_i$. 
Next, if $m\ge3$, from $i=1$ to $i=m-2$, we perform the following $(m-i-1)$-times elementary row operations: 
\begin{itemize}

\item for $i+1\le j\le m-1$, we multiply the $i$-th row by $\ell_{j+1,i+1}d_{i+1}/d_j$ and add it to the $j$-th row. 

\end{itemize}
Finally, for $1\le i\le m-1$, we multiply the $i$-th row by $-b_{i+1}d_{i+1}$ and add it to the $m$-th row. 
Then the matrix $D$ transforms into 
\begin{equation}
E=\left(\begin{matrix}-r_2 & a_1/d_2 & 0 & \cdots & 0 \\
-r_3 & 0 & a_1/d_3 & \cdots & 0 \\
\vdots & \vdots & \vdots & \ddots & \vdots \\
-r_m & 0 &0 & \cdots &  a_1 \\
r_1 & 0 & 0& \cdots  & 0
\end{matrix}\right),\label{2024.1.6.456}
\end{equation}
where $r_2=\ell_{2,1}$, $r_3,\dots,r_m\in\mathbb{Z}_{\ge0}$, and $r_1\in\mathbb{Z}$. 
For $2\le i\le m$, we prove 
\begin{equation}
p_{i,0}^{a_1/d_i}=p_{1,0}^{r_i}.\label{2024.1.6.1}
\end{equation}
For $i=2$, from (\ref{p1}), the relation (\ref{2024.1.6.1}) holds. 
We take an integer $k$ such that $3\le k\le m$. 
For any $2\le i\le k-1$, we assume that the relation (\ref{2024.1.6.1}) holds.
We have 
\[r_k=\ell_{k,1}\frac{a_1}{d_{k-1}}+\sum_{j=2}^{k-1}r_j\ell_{k,j}\frac{d_j}{d_{k-1}}.\]
From (\ref{p1}), we have 
\begin{align*}
p_{k,0}^{a_1/d_k}&=\prod_{j=1}^{k-1}p_{j,0}^{\ell_{k,j}a_1/d_{k-1}}=p_{1,0}^{\ell_{k,1}a_1/d_{k-1}}\prod_{j=2}^{k-1}\left(p_{j,0}^{a_1/d_j}\right)^{\ell_{k,j}d_j/d_{k-1}}\\
&=p_{1,0}^{\ell_{k,1}a_1/d_{k-1}}\prod_{j=2}^{k-1}\left(p_{1,0}^{r_j}\right)^{\ell_{k,j}d_j/d_{k-1}}=p_{1,0}^{r_k}.
\end{align*}
Thus, for $i=k$, the relation (\ref{2024.1.6.1}) holds. 
By induction, the relation (\ref{2024.1.6.1}) holds for any $2\le i\le m$. 
We have 
\[r_1=a_1b_1+\sum_{i=2}^mr_ib_id_i.\]
From (\ref{2024.1.6.1}), we have 
\[\prod_{i=1}^mp_{i,0}^{a_1b_i}=p_{1,0}^{a_1b_1}\prod_{i=2}^mp_{i,0}^{a_1b_i}=p_{1,0}^{a_1b_1}\prod_{i=2}^mp_{1,0}^{r_ib_id_i}=p_{1,0}^{r_1}.\]
From (\ref{p2}), we have $p_{1,0}^{r_1}=1$. 
From Lemma \ref{lem11}, we have $\det(E)=(-1)^m\prod_{i=1}^ma_1/d_i$. 
On the other hand, from (\ref{2024.1.6.456}), we have $\det(E)=(-1)^{m-1}r_1\prod_{i=2}^ma_1/d_i$. 
Thus, we have $r_1=-1$. Therefore, we have $p_{1,0}=1$. 
Let $\zeta_{a_1}$ be a primitive $a_1$-th root of unity. 
From (\ref{2024.1.6.1}), there exist integers $k_2,\dots,k_m$ such that $1\le k_i\le a_1$ and $p_{i,0}=\zeta_{a_1}^{k_i}$ for $2\le i\le m$. 
We set $k_1=a_1$. 
From (\ref{p1}), we have 
\begin{equation}
\zeta_{a_1}^{k_id_{i-1}/d_i}=\zeta_{a_1}^{\sum_{j=1}^{i-1}k_j\ell_{i,j}}\label{2023.8.31.1}
\end{equation}
for $2\le i\le m$. 
From (\ref{p2}), we have 
\begin{equation}
\zeta_{a_1}^{k_1b_1+\cdots+k_mb_m}=1.\label{2023.8.31.2}
\end{equation}
From (\ref{2023.8.31.1}) and (\ref{2023.8.31.2}), we have 
\[D\left(\begin{matrix}k_1 \\ \vdots \\ k_m\end{matrix}\right)\in a_1\mathbb{Z}^m.\]
From Lemma \ref{lem11}, we have $D^{-1}\in M_m(\mathbb{Z})$, where $D^{-1}$ is the inverse matrix of $D$. 
Therefore, we have $k_i\in a_1\mathbb{Z}$ for any $1\le i\le m$. 
From $1\le k_i\le a_1$, we have $k_i=a_1$ for any $1\le i\le m$. 
Thus, we obtain $p_{i,0}=1$ for any $1\le i\le m$. 
\end{proof}

From (\ref{t}), we have $\deg t=-1$. 

\begin{prop}\label{prop12}
We have $p_{i,k}\in\mathbb{Z}[{\boldsymbol \l}]$ and the expansion of $x_i$ around $\infty$ with respect to $t$ is homogeneous of degree $a_i$ with respect to ${\boldsymbol \l}$ and $t$. 
\end{prop}

\begin{proof}
From Proposition \ref{prop11}, for any $1\le i\le m$, we have $p_{i,0}\in\mathbb{Z}[{\boldsymbol \l}]$ and $p_{i,0}$ is homogeneous of degree $0$ with respect to ${\boldsymbol \l}$.  
We take an integer $\ell\ge1$. 
For any $1\le i\le m$ and $0\le k\le \ell-1$, we assume that $p_{i,k}\in\mathbb{Z}[{\boldsymbol \l}]$ and $p_{i,k}$ is homogeneous of degree $k$ with respect to ${\boldsymbol \l}$ if $p_{i,k}\neq0$. 
By comparing the coefficients of $t^{\ell}$ in (\ref{ar}), there exist $J_i({\boldsymbol \l})\in\mathbb{Z}[{\boldsymbol \l}]$ for $2\le i\le m$ such that $J_i({\boldsymbol \l})$ is homogeneous of degree $\ell$ with respect to ${\boldsymbol \l}$ if $J_i({\boldsymbol \l})\neq0$ and 
\[\frac{d_{i-1}}{d_i}p_{i,\ell}-\sum_{j=1}^{i-1}\ell_{i,j}p_{j,\ell}=J_i({\boldsymbol \l}),\;\;\;\;\;\;2\le i\le m.\]
By comparing the coefficients of $t^{\ell}$ in (\ref{kuro}), there exists $J_1({\boldsymbol \l})\in\mathbb{Z}[{\boldsymbol \l}]$ such that 
$J_1({\boldsymbol \l})$ is homogeneous of degree $\ell$ with respect to ${\boldsymbol \l}$ if $J_1({\boldsymbol \l})\neq0$ and 
\[b_1p_{1,\ell}+\cdots+b_mp_{m,\ell}=J_1({\boldsymbol \l}).\]
Therefore, we obtain 
\[D\left(\begin{matrix}p_{1,\ell} \\ \vdots \\ p_{m,\ell}\end{matrix}\right)=\left(\begin{matrix} J_2({\boldsymbol \l})\\ \vdots \\ J_m({\boldsymbol \l})\\ J_1({\boldsymbol \l})\end{matrix}\right).\]
Since $D^{-1}\in M_m(\mathbb{Z})$, for any $1\le i\le m$, we have $p_{i,\ell}\in\mathbb{Z}[{\boldsymbol \l}]$ and $p_{i,\ell}$ is homogeneous of degree $\ell$ with respect to ${\boldsymbol \l}$ if $p_{i,\ell}\neq0$. 
By induction, for any $1\le i\le m$ and $k\ge0$, we have $p_{i,k}\in\mathbb{Z}[{\boldsymbol \l}]$ and $p_{i,k}$ is homogeneous of degree $k$ with respect to ${\boldsymbol \l}$ if $p_{i,k}\neq0$. 
\end{proof}

\begin{lem}[{\cite[Lemma 3.4]{Aya3}}]\label{gkp}
For $1\le k\le m$, we have
\begin{equation}
\det G_k(P)=(-1)^{k+1}a_kx_1^{\gamma_1}\cdots x_m^{\gamma_m}+\sum \beta_{i_1,\dots,i_m}x_1^{i_1}\cdots x_m^{i_m},\label{2022.6.28.1}
\end{equation}
where $(\gamma_1,\dots,\gamma_m)$ is the unique element of $B(A_m)$ such that $\sum_{j=1}^ma_j\gamma_j=2g-1+a_k$ and 
the summation in (\ref{2022.6.28.1}) is taken over $(i_1,\dots,i_m)\in B(A_m)$ such that $\sum_{j=1}^ma_ji_j<2g-1+a_k$. 
We have $\beta_{i_1,\dots,i_m}\in\mathbb{Z}[{\boldsymbol \l}]$ and the right hand side of (\ref{2022.6.28.1}) is homogeneous of degree $2g-1+a_k$ with respect to ${\boldsymbol \l}$ and $x_1,\dots,x_m$. 
\end{lem}

In the same way as \cite[Proposition 1]{Aya2}, for $1\le i\le g$, we can prove that the expansion of $\omega_i$ around $\infty$ with respect to $t$ has the following form: 
\begin{equation}
\omega_i=t^{w_i-1}\left(1+\sum_{j=1}^{\infty}b_{i,j}t^j\right)dt,\;\;\;\;b_{i,j}\in\mathbb{C}.\label{2023.10.1}
\end{equation}

\begin{prop}\label{2021.8.8.1}
We have $b_{i,j}\in\mathbb{Z}[{\boldsymbol \l}]$ and the expansion of $\omega_i$ around $\infty$ with respect to $t$ is homogeneous of degree $1-w_i$ with respect to ${\boldsymbol \l}$ and $t$. 
\end{prop}

\begin{proof}
From Propositions \ref{prop11}, \ref{prop12}, and Lemma \ref{gkp}, for $1\le k\le m$, we have the following expansion: 
\[\frac{dx_k}{\det G_k(P)}=(-1)^kt^{2g-2}\left(1+\sum_{j=1}^{\infty}b_j^{(k)}t^j\right)dt,\]
where $b_j^{(k)}\in\mathbb{Z}[1/a_k,{\boldsymbol \l}]$ and $b_j^{(k)}$ is homogeneous of degree $j$ with respect to ${\boldsymbol \l}$ if $b_j^{(k)}\neq0$. 
For any $1\le k\le m$, we have 
\begin{equation}
\frac{dx_1}{\det G_1(P)}=(-1)^{k-1}\frac{dx_k}{\det G_k(P)}.\label{2022.6.2.3}
\end{equation}
(cf. the proof of \cite[Lemma 3.2]{Aya1}). 
Therefore, we have $b_j^{(1)}=b_j^{(k)}$ for any $j\ge1$ and $2\le k\le m$. 
Since $\mathrm{gcd}(a_1, \dots, a_m)=1$, for any $j\ge1$, we have 
\[b_j^{(1)}\in\bigcap_{k=1}^m\mathbb{Z}[1/a_k,{\boldsymbol \l}]=\mathbb{Z}[{\boldsymbol \l}].\]
From Propositions \ref{prop11} and \ref{prop12}, we obtain the statement of the proposition. 
\end{proof}

We take $\{\eta_i\}_{i=1}^g$ as in Lemma \ref{2022.8.6.1234}. 

\begin{prop}\label{ec}
It is possible to take $\{\eta_i\}_{i=1}^g$ such that $c_{i_1,\dots,i_m;j_1,\dots,j_m}=0$ if 
$\sum_{k=1}^ma_ki_k\ge\sum_{k=1}^ma_kj_k$. 
\end{prop}

\begin{proof}
If $\sum_{k=1}^ma_ki_k>\sum_{k=1}^ma_kj_k$ and $c_{i_1,\dots,i_m;j_1,\dots,j_m}\neq0$, we add 
\[-\frac{c_{i_1,\dots,i_m;j_1,\dots,j_m}x_1^{i_1}\cdots x_m^{i_m}y_1^{j_1}\cdots y_m^{j_m}}{\det G_1(P)\det G_1(Q)}dx_1dy_1-\frac{c_{i_1,\dots,i_m;j_1,\dots,j_m}x_1^{j_1}\cdots x_m^{j_m}y_1^{i_1}\cdots y_m^{i_m}}{\det G_1(P)\det G_1(Q)}dx_1dy_1\]
to $\widehat{\omega}(P,Q)$. 
If $\sum_{k=1}^ma_ki_k=\sum_{k=1}^ma_kj_k$, which is equivalent to $(i_1,\dots,i_m)=(j_1,\dots,j_m)$, and $c_{i_1,\dots,i_m;j_1,\dots,j_m}\neq0$, we add 
\[-\frac{c_{i_1,\dots,i_m;i_1,\dots,i_m}x_1^{i_1}\cdots x_m^{i_m}y_1^{i_1}\cdots y_m^{i_m}}{\det G_1(P)\det G_1(Q)}dx_1dy_1\]
to $\widehat{\omega}(P,Q)$. 
Then we can take $\{\eta_i\}_{i=1}^g$ in the form as claimed. 
\end{proof}

Hereafter, we take $\{\eta_i\}_{i=1}^g$ as in Proposition \ref{ec}.  

\begin{lem}[{\cite[p.~470]{Aya1}, \cite[p.~6]{Aya3}}]\label{dqo}
We have 
\begin{align*}
&d_Q\Omega(P,Q) \\
&=\frac{\{\sum_{i=1}^m(-1)^{i+1}(x_1-y_1)\frac{\partial \det H}{\partial y_i}(P,Q)\det G_i(Q)\}+\det G_1(Q)\det H(P,Q)}{(x_1-y_1)^2\det G_1(P)\det G_1(Q)}dx_1dy_1,
\end{align*}
where the numerator is homogeneous of degree $2\sum_{i=2}^m(d_{i-1}/d_i-1)a_i=2(2g-1+a_1)$ 
with respect to ${\boldsymbol \l}$ and $x_1,\dots,x_m, y_1,\dots,y_m$.  
\end{lem}

We define $\overline{c}_{i_1,\dots,i_m;j_1,\dots,j_m}$ by 
\[d_Q\Omega(P,Q)=\frac{\sum\overline{c}_{i_1,\dots,i_m;j_1,\dots,j_m}x_1^{i_1}\cdots x_m^{i_m}y_1^{j_1}\cdots y_m^{j_m}}{(x_1-y_1)^2\det G_1(P)\det G_1(Q)}dx_1dy_1,\]
where the summation is taken over $(i_1,\dots,i_m), (j_1,\dots,j_m)\in B(A_m)$. 

\begin{lem}\label{2021.8.6.1}
We have $\overline{c}_{i_1,\dots,i_m;j_1,\dots,j_m}\in\mathbb{Z}[{\boldsymbol \l}]$ 
and $\overline{c}_{i_1,\dots,i_m;j_1,\dots,j_m}$ is homogeneous of degree $2(2g-1+a_1)-\sum_{k=1}^ma_k(i_k+j_k)$ with respect to ${\boldsymbol \l}$ if $\overline{c}_{i_1,\dots,i_m;j_1,\dots,j_m}\neq0$. 
\end{lem}

\begin{proof}
From Lemmas \ref{2022.6.28.222} and \ref{dqo}, we obtain the statement of the lemma.
\end{proof}


We define $F(P,Q)$ by 
\[\widehat{\omega}(P,Q)=\frac{F(P,Q)}{(x_1-y_1)^2\det G_1(P)\det G_1(Q)}dx_1dy_1.\]

We can determine the coefficients $c_{i_1,\dots,i_m;j_1,\dots,j_m}$ by the following recurrence relations. 

\begin{prop}\label{i012}
We take $(i_1,\dots,i_m), (j_1,\dots,j_m)\in B(A_m)$ such that $\sum_{k=1}^ma_ki_k<\sum_{k=1}^ma_kj_k$. 

\vspace{1ex}

(i) If $i_1=0$, we have 
\[c_{0,i_2,\dots,i_m;j_1,\dots,j_m}=\overline{c}_{j_1+2,j_2,\dots,j_m;0,i_2,\dots,i_m}-\overline{c}_{0,i_2,\dots,i_m;j_1+2,j_2,\dots,j_m}.\]

\vspace{1ex}

(ii) If $i_1=1$, we have 
\begin{align*}
c_{1,i_2,\dots,i_m;j_1,\dots,j_m}&=2\overline{c}_{j_1+3,j_2,\dots,j_m;0,i_2,\dots,i_m}-2\overline{c}_{0,i_2,\dots,i_m;j_1+3,j_2,\dots,j_m}+\overline{c}_{j_1+2,j_2,\dots,j_m;1,i_2,\dots,i_m} \\ 
&-\overline{c}_{1,i_2,\dots,i_m;j_1+2,j_2,\dots,j_m}. 
\end{align*}

\vspace{1ex}

(iii) If $i_1\ge2$, we have 
\begin{align*}
c_{i_1,\dots,i_m;j_1,\dots,j_m}&=2c_{i_1-1,i_2,\dots,i_m;j_1+1,j_2,\dots,j_m}-c_{i_1-2,i_2,\dots,i_m;j_1+2,j_2,\dots,j_m} \\
&+\overline{c}_{j_1+2,j_2,\dots,j_m;i_1,\dots,i_m}-\overline{c}_{i_1,\dots,i_m;j_1+2,j_2,\dots,j_m}.  
\end{align*}

\end{prop}

\begin{proof}
(i) The coefficient of $x_2^{i_2}\cdots x_m^{i_m}y_1^{j_1+2}y_2^{j_2}\cdots y_m^{j_m}$ in $F(P,Q)$ is 
\[c_{0,i_2,\dots,i_m;j_1,\dots,j_m}+\overline{c}_{0,i_2,\dots,i_m;j_1+2,j_2,\dots,j_m}.\]
From $\sum_{k=2}^ma_ki_k<\sum_{k=1}^ma_kj_k$ and Proposition \ref{ec}, the coefficient of $x_1^{j_1+2}x_2^{j_2}\cdots x_m^{j_m}y_2^{i_2}\cdots y_m^{i_m}$ in $F(P,Q)$ is $\overline{c}_{j_1+2,j_2,\dots,j_m;0,i_2,\dots,i_m}$. 
From $\widehat{\omega}(Q,P)=\widehat{\omega}(P,Q)$, we obtain the statement of (i). 

\vspace{1ex}

(ii) The coefficient of $x_1x_2^{i_2}\cdots x_m^{i_m}y_1^{j_1+2}y_2^{j_2}\cdots y_m^{j_m}$ in $F(P,Q)$ is 
\[c_{1,i_2,\dots,i_m;j_1,\dots,j_m}-2c_{0,i_2,\dots,i_m;j_1+1,j_2,\dots,j_m}+\overline{c}_{1,i_2,\dots,i_m;j_1+2,j_2,\dots,j_m}.\]
From $a_1+\sum_{k=2}^ma_ki_k<\sum_{k=1}^ma_kj_k$ and Proposition \ref{ec}, the coefficient of $x_1^{j_1+2}x_2^{j_2}\cdots x_m^{j_m}y_1y_2^{i_2}\cdots y_m^{i_m}$ in $F(P,Q)$ is $\overline{c}_{j_1+2,j_2,\dots,j_m;1,i_2,\dots,i_m}$. 
From $\widehat{\omega}(Q,P)=\widehat{\omega}(P,Q)$ and (i), we obtain the statement of (ii).  

\vspace{1ex}

(iii) The coefficient of $x_1^{i_1}\cdots x_m^{i_m}y_1^{j_1+2}y_2^{j_2}\cdots y_m^{j_m}$ in $F(P,Q)$ is 
\[c_{i_1,\dots,i_m;j_1,\dots,j_m}-2c_{i_1-1,i_2,\dots,i_m;j_1+1,j_2,\dots,j_m}+c_{i_1-2,i_2,\dots,i_m;j_1+2,j_2,\dots,j_m}+\overline{c}_{i_1,\dots,i_m;j_1+2,j_2,\dots,j_m}.\]
From $\sum_{k=1}^ma_ki_k<\sum_{k=1}^ma_kj_k$ and Proposition \ref{ec}, the coefficient of $x_1^{j_1+2}x_2^{j_2}\cdots x_m^{j_m}y_1^{i_1}\cdots y_m^{i_m}$ in $F(P,Q)$ is $\overline{c}_{j_1+2,j_2,\dots,j_m;i_1,\dots,i_m}$. 
From $\widehat{\omega}(Q,P)=\widehat{\omega}(P,Q)$, we obtain the statement of (iii). 
\end{proof}



\begin{lem}\label{2021.8.6.2}
We have $c_{i_1,\dots,i_m;j_1,\dots,j_m}\in\mathbb{Z}[{\boldsymbol \l}]$ and 
$c_{i_1,\dots,i_m;j_1,\dots,j_m}$ is homogeneous of degree $2(2g-1)-\sum_{k=1}^ma_k(i_k+j_k)$ with respect to ${\boldsymbol \l}$ if $c_{i_1,\dots,i_m;j_1,\dots,j_m}\neq0$. 
\end{lem}

\begin{proof}
From Proposition \ref{i012}, we obtain the statement of the lemma. 
\end{proof}

We define $\widehat{c}_{i_1,\dots,i_m;j_1,\dots,j_m}$ by 
\[
F(P,Q)=\sum\widehat{c}_{i_1,\dots,i_m;j_1,\dots,j_m}x_1^{i_1}\cdots x_m^{i_m}y_1^{j_1}\cdots y_m^{j_m},
\]
where the summation is taken over $(i_1,\dots,i_m), (j_1,\dots,j_m)\in B(A_m)$. 

\begin{lem}\label{2021.8.6.3}
We have $\widehat{c}_{i_1,\dots,i_m;j_1,\dots,j_m}\in\mathbb{Z}[{\boldsymbol \l}]$ 
and $\widehat{c}_{i_1,\dots,i_m;j_1,\dots,j_m}$ is homogeneous of degree $2(2g-1+a_1)-\sum_{k=1}^ma_k(i_k+j_k)$ with respect to ${\boldsymbol \l}$ if $\widehat{c}_{i_1,\dots,i_m;j_1,\dots,j_m}\neq0$. 
\end{lem}

\begin{proof}
From Lemmas \ref{2021.8.6.1} and \ref{2021.8.6.2}, we obtain the statement of the lemma.
\end{proof}

The Klein's fundamental $2$-form $\widehat{\omega}(P,Q)$ is expanded around $\infty\times\infty$ with respect to the arithmetic local parameter $t$ as follows: 
\begin{equation}
\widehat{\omega}(P,Q)=\left(\frac{1}{(t_P-t_Q)^2}+\sum_{i,j\ge1}q_{i,j}t_P^{i-1}t_Q^{j-1}\right)dt_Pdt_Q,\;\;\;\;q_{i,j}\in\mathbb{C},\label{2021.8.6.4}
\end{equation}
where $t_P=t(P)$ and $t_Q=t(Q)$. 
From $\widehat{\omega}(Q,P)=\widehat{\omega}(P,Q)$, we have $q_{j,i}=q_{i,j}$ for any $i,j$. 

\begin{prop}\label{2022.3.20.1}
We have $q_{i,j}\in\mathbb{Z}[{\boldsymbol \l}]$ 
and $q_{i,j}$ is homogeneous of degree $i+j$ with respect to ${\boldsymbol \l}$ if $q_{i,j}\neq0$. 
\end{prop}

\begin{proof}
From (\ref{xi}) and Proposition \ref{prop11}, we have 
\begin{equation}
t_P^{2a_1}t_Q^{2a_1}(x_1-y_1)^2=\left(t_Q^{a_1}+\sum_{k=1}^{\infty}p_{1,k}t_P^kt_Q^{a_1}-t_P^{a_1}-\sum_{k=1}^{\infty}p_{1,k}t_P^{a_1}t_Q^{k}\right)^2.\label{2022.3.19.111}
\end{equation}
From Proposition \ref{prop12}, we have 
\begin{equation}
t_P^{2a_1}t_Q^{2a_1}(x_1-y_1)^2=(t_P-t_Q)^2\sum_{i,j\ge0}\nu_{i,j}t_P^it_Q^j,\label{2022.3.19.1}
\end{equation}
where $\nu_{i,j}\in\mathbb{Z}[{\boldsymbol \l}]$ and $\nu_{i,j}$ is homogeneous of degree $2-2a_1+i+j$ with respect to ${\boldsymbol \l}$ if $\nu_{i,j}\neq0$. 
From (\ref{2021.8.6.4}), (\ref{2022.3.19.111}), and (\ref{2022.3.19.1}), around $\infty\times\infty$, we have 
\begin{align}
&t_P^{2a_1}t_Q^{2a_1}(x_1-y_1)^2\widehat{\omega}(P,Q)=\nonumber\\
&\left\{\sum_{i,j\ge0}\nu_{i,j}t_P^it_Q^j+\left(t_Q^{a_1}+\sum_{k=1}^{\infty}p_{1,k}t_P^kt_Q^{a_1}-t_P^{a_1}-\sum_{k=1}^{\infty}p_{1,k}t_P^{a_1}t_Q^{k}\right)^2\left(\sum_{i,j\ge1}q_{i,j}t_P^{i-1}t_Q^{j-1}\right)\right\}dt_Pdt_Q.\label{2021.8.7.111}
\end{align}
Therefore, around $\infty\times\infty$, we have 
\begin{equation}
t_P^{2a_1}t_Q^{2a_1}(x_1-y_1)^2\widehat{\omega}(P,Q)=\left(\sum_{i,j\ge0}\widehat{q}_{i,j}t_P^it_Q^j\right)dt_Pdt_Q,\;\;\;\;\widehat{q}_{i,j}\in\mathbb{C}.\label{2022.3.19.11111}
\end{equation}
From Propositions \ref{prop12}, \ref{2021.8.8.1}, and Lemma \ref{2021.8.6.3}, we have $\widehat{q}_{i,j}\in\mathbb{Z}[{\boldsymbol \l}]$ and $\widehat{q}_{i,j}$ is homogeneous of degree $2-2a_1+i+j$ with respect to ${\boldsymbol \l}$ if $\widehat{q}_{i,j}\neq0$. 
From (\ref{2021.8.7.111}) and (\ref{2022.3.19.11111}), we have 
\begin{align}
&\sum_{i,j\ge0}(\widehat{q}_{i,j}-\nu_{i,j})t_P^it_Q^j\nonumber\\
&=\left(t_Q^{a_1}+\sum_{k=1}^{\infty}p_{1,k}t_P^kt_Q^{a_1}-t_P^{a_1}-\sum_{k=1}^{\infty}p_{1,k}t_P^{a_1}t_Q^{k}\right)^2\left(\sum_{i,j\ge1}q_{i,j}t_P^{i-1}t_Q^{j-1}\right).\label{2021.8.7.1}
\end{align}
By comparing the coefficients of $t_Q^{2a_1}$ in the both sides of (\ref{2021.8.7.1}), we have $q_{1,1}\in\mathbb{Z}[{\boldsymbol \l}]$ and $q_{1,1}$ is homogeneous of degree $2$ with respect to ${\boldsymbol \l}$ if $q_{1,1}\neq0$. 
We take a pair of positive integers $(i_0,j_0)$. 
For any $(i,j)\in\mathbb{N}^2$ such that 

\begin{itemize}
\item $i+j< i_0+j_0$ or 

\item $i+j=i_0+j_0$ and $i<i_0$, 
\end{itemize}
we assume that $q_{i,j}\in\mathbb{Z}[{\boldsymbol \l}]$ 
and $q_{i,j}$ is homogeneous of degree $i+j$ with respect to ${\boldsymbol \l}$ if $q_{i,j}\neq0$.  
By comparing the coefficients of $t_P^{i_0-1}t_Q^{j_0+2a_1-1}$ in the both sides of (\ref{2021.8.7.1}), we have $q_{i_0,j_0}\in\mathbb{Z}[{\boldsymbol \l}]$ and $q_{i_0,j_0}$ is homogeneous of degree $i_0+j_0$ with respect to ${\boldsymbol \l}$ if $q_{i_0,j_0}\neq0$. 
By induction, we obtain the statement of the proposition. 
\end{proof}

\section{Hurwitz integrality of the power series expansion of the sigma function for the telescopic curve}

\begin{defn}
For a subring $R$ of $\mathbb{C}$ and variables $z={}^t(z_1,\dots, z_n)$, let
\[R\langle\langle z \rangle\rangle=R\langle\langle z_1,\dots,z_n\rangle\rangle=\left\{\sum_{i_1,\dots,i_n\ge0}\kappa_{i_1,\dots,i_n}\frac{z_1^{i_1}\cdots z_n^{i_n}}{i_1!\cdots i_n!}\;\middle|\;\kappa_{i_1,\dots,i_n}\in R\right\}.\]
If the power series expansion of a holomorphic function $f(z)=f(z_1,\dots,z_n)$ on $\mathbb{C}^n$ around the origin belongs to $R\langle\langle z \rangle\rangle$, then we write $f(z)\in R\langle\langle z \rangle\rangle$ and $f(z)$ is said to be \textit{Hurwitz integral} over $R$. 
\end{defn}

Let $W=\{w_1,\dots,w_g\}$ and $u={}^t(u_1, \dots, u_g)$. 
For any partition $\mu$ and the Schur function $S_{\mu}(T)$, we substitute $T_{w_i}=u_i$ for $1\le i\le g$ and $T_j=0$ for any $j$ satisfying $j\notin W$, and denote it by $S_{\mu}(u)$. 

\begin{lem}\la{schur}
For any partition $\mu$, we have $S_{\mu}(u)\in\mathbb{Z}\langle\langle u \rangle\rangle$.
\end{lem}

\begin{proof}
For an integer $n\ge1$ and the polynomial $p_n(T)$ (cf. Section \ref{2022.3.19.9}), we substitute $T_{w_i}=u_i$ for $1\le i\le g$ and $T_j=0$ for any $j$ satisfying $j\notin W$, and denote it by $p_n(u)$. 
Let $p_0(u)=1$ and $p_n(u)=0$ for $n<0$. 
From Lemma \ref{2022.3.19.99}, for $n\ge0$, we have 
\[p_n(u)=\sum\frac{u_1^{n_1}\cdots u_g^{n_g}}{n_1!\cdots n_g!},\]
where the summation is taken over $(n_1,\dots,n_g)\in\mathbb{Z}_{\ge0}^g$ satisfying $w_1n_1+\cdots+w_gn_g=n$. 
We have 
\[S_{\mu}(u)=\det \left(p_{\mu_i-i+j}(u)\right)_{1\le i,j\le \ell},\]
where $\mu=(\mu_1,\mu_2,\dots,\mu_{\ell})$. For integers $m_1,\dots,m_g, n_1,\dots,n_g\ge0$, we have
\[\frac{u_1^{m_1}\cdots u_g^{m_g}}{m_1!\cdots m_g!}\frac{u_1^{n_1}\cdots u_g^{n_g}}{n_1!\cdots n_g!}=\left(\begin{matrix}m_1+n_1\\m_1\end{matrix}\right)
\cdots\left(\begin{matrix}m_g+n_g\\m_g\end{matrix}\right)\frac{u_1^{m_1+n_1}\cdots u_g^{m_g+n_g}}{(m_1+n_1)!\cdots(m_g+n_g)!}.\]
Since the binomial coefficients $\left(\begin{matrix}m_1+n_1\\m_1\end{matrix}\right),\dots,\left(\begin{matrix}m_g+n_g\\m_g\end{matrix}\right)$ are integers, we obtain the statement of the lemma.
\end{proof}

\begin{lem}\label{2022.3.20.888}
Let $R$ be a subring of $\mathbb{C}$, $f(u)=f(u_1,\dots,u_g)$ be a holomorphic function on $\mathbb{C}^g$, and $M\in M_g(R)$. 
If $f(u)\in R\langle\langle u \rangle\rangle$, then we have $f(Mu)\in R\langle\langle u \rangle\rangle$. 
\end{lem}

\begin{proof}
Let $M=(m_{i,j})_{1\le i,j\le g}$, where $m_{i,j}\in R$. 
For any integer $n\ge0$, we have 
\[\frac{(m_{i,1}u_1+\cdots+m_{i,g}u_g)^n}{n!}=\sum m_{i,1}^{n_1}\cdots m_{i,g}^{n_g}\frac{u_1^{n_1}\cdots u_g^{n_g}}{n_1!\cdots n_g!},\]
where the summation is taken over $(n_1,\dots,n_g)\in\mathbb{Z}_{\ge0}^g$ satisfying $n_1+\cdots+n_g=n$. 
Thus, we obtain the statement of the lemma. 
\end{proof}

We expand $t^{g-1}\varphi_j$ around $\infty$ with respect to the arithmetic local parameter $t$ 
\[t^{g-1}\varphi_j=\sum_i\xi_{i,j}t^i.\]
From Proposition \ref{prop12}, we have $\xi_{i,j}\in\mathbb{Z}[{\boldsymbol \l}]$. 
For $j>g$, we have 
\[\xi_{i,j}=\left\{\begin{array}{ll} 0 & \;\;\mbox{if}\;\;i<-j \\ 1 & \;\;\mbox{if} \;\;i=-j.\end{array}\right.\]
For a partition $\mu=(\mu_1,\mu_2,\dots)$, we define
\[\xi_{\mu}=\det(\xi_{m_i,j})_{i,j\in\mathbb{N}}=\left|\begin{matrix}\xi_{m_1,1} &\xi_{m_1,2}&\xi_{m_1,3} & \cdots\\ \xi_{m_2,1}&\xi_{m_2,2}&\xi_{m_2,3}&\cdots\\  \xi_{m_3,1}&\xi_{m_3,2}&\xi_{m_3,3}&\cdots\\\vdots&\vdots&\vdots&\ddots\end{matrix}\right|,\]
where $m_i=\mu_i-i$ and the infinite determinant is well defined.
Then we have $\xi_{\mu}\in\mathbb{Z}[{\boldsymbol \l}]$.
The tau function $\tau(u)$ is defined by
\[\tau(u)=\sum_{\mu}\xi_{\mu}S_{\mu}(u),\]
where the summation is taken over all partitions.

\begin{prop}\la{tauhe}
We have $\tau(u)\in\mathbb{Z}[{\boldsymbol \l}]\langle\langle u \rangle\rangle$.
\end{prop}

\begin{proof}
From Lemma \ref{schur}, we obtain the statement of the proposition. 
\end{proof}

For $k\ge1$, we define $c_k$ by 
\begin{equation}
\sum_{k=1}^{\infty}c_kt^{k-1}=\frac{1}{2}\frac{\frac{d}{dt}\left(1+\sum_{j=1}^{\infty}b_{g,j}t^j\right)}{1+\sum_{j=1}^{\infty}b_{g,j}t^j},\label{2022.3.20.22}
\end{equation}
where $b_{g,j}$ are defined in (\ref{2023.10.1}). 
For $1\le i\le g$, we expand $\omega_i$ around $\infty$ with respect to $t$ as follows: 
\[\omega_i=\sum_{j=1}^{\infty}\widehat{b}_{i,j}t^{j-1}dt,\;\;\;\;\;\widehat{b}_{i,j}\in\mathbb{C}.\]
From Proposition \ref{2021.8.8.1}, we have $\widehat{b}_{i,j}\in\mathbb{Z}[{\boldsymbol \l}]$ and 
\[\widehat{b}_{i,j}=\left\{\begin{array}{ll} 0 &\;\; \mbox{if}\;\;j<w_i \\ 1 & \;\;\mbox{if} \;\;j=w_i\end{array}\right..\]
We define the $g\times g$ matrix
\[B=(\widehat{b}_{i,w_j})_{1\le i, j\le g}=\left(\begin{matrix}1&\widehat{b}_{1,w_2}&\widehat{b}_{1,w_3}&\cdots&\widehat{b}_{1,w_g}\\0&1&\widehat{b}_{2,w_3}&\cdots & \widehat{b}_{2,w_g}\\ 0&0&1&\cdots &\widehat{b}_{3,w_g}\\\vdots &\vdots &\vdots &\ddots &\vdots \\ 0&0&0&\cdots& 1\end{matrix}\right),\]
$c=(c_{w_1},\dots,c_{w_g})$, and the $g\times g$ matrix $N=(q_{w_i,w_j})_{1\le i,j\le g}$. 

\begin{thm}[{\cite[Theorem 1]{Aya2}, \cite[Theorem 8]{N-2010-2}}]\label{2021.9.9}
For $v={}^t(v_1,\dots,v_g)\in\mathbb{C}^g$, the following relation holds: 
\[\tau(v)=\exp\left(-cv+\frac{1}{2}{}^tvNv\right)\sigma(Bv).\]
\end{thm}

\vspace{1ex}

For a subset $\mathfrak{S}$ of ${\boldsymbol \lambda}$, 
we set $\lambda_{j, \mathfrak{S}}^{(i)}=\lambda_j^{(i)}/2$ if $\lambda_j^{(i)}\in \mathfrak{S}$ and we set $\lambda_{j, \mathfrak{S}}^{(i)}=\lambda_j^{(i)}$ 
if $\lambda_j^{(i)}\notin \mathfrak{S}$. 
We denote by ${\boldsymbol \l}_{\mathfrak{S}}$ the set of all $\lambda_{j, \mathfrak{S}}^{(i)}$. 
We set $\deg\lambda_{j, \mathfrak{S}}^{(i)}=\deg\lambda_j^{(i)}$. 

\begin{ex}
We consider the $(2,3)$ curve. 
The polynomial $F_2$ defining the $(2,3)$ curve is given by
\[F_2(X)=X_2^2-X_1^3-\lambda_1X_1X_2-\lambda_2X_1^2-\lambda_3X_2-\lambda_4X_1-\lambda_6.\] 
We have ${\boldsymbol \lambda}=\{\lambda_1, \lambda_2, \lambda_3, \lambda_4, \lambda_6\}$. 
We consider the case of $\mathfrak{S}=\{\lambda_1, \lambda_3\}$. 
Then we have 
\[\lambda_{1,\mathfrak{S}}=\lambda_1/2,\quad \lambda_{2,\mathfrak{S}}=\lambda_2,\quad \lambda_{3,\mathfrak{S}}=\lambda_3/2,\quad \lambda_{4,\mathfrak{S}}=\lambda_4,\quad\lambda_{6,\mathfrak{S}}=\lambda_6\]
and ${\boldsymbol \lambda}_{\mathfrak{S}}=\{\lambda_{1,\mathfrak{S}}, \lambda_{2,\mathfrak{S}}, \lambda_{3,\mathfrak{S}}, \lambda_{4,\mathfrak{S}}, \lambda_{6,\mathfrak{S}}\}$. 
\end{ex}

Let $\mathfrak{A}=\{\lambda_j^{(i)}\;|\;\mbox{$j$ is odd}\}$. 
For a domain $R$, we denote by $R[[t]]$ the set consisting of formal power series over $R$. 

\begin{lem}\label{2021.9.10.33}
We have $c_k\in\mathbb{Z}[{\boldsymbol \l}_{\mathfrak{A}}]$ 
and $c_k$ is homogeneous of degree $k$ with respect to ${\boldsymbol \l}_{\mathfrak{A}}$ if $c_k\neq0$. 
\end{lem}

\begin{proof}
From Proposition \ref{2021.8.8.1}, we have $b_{g,j}\in\mathbb{Z}[{\boldsymbol \l}]$ and $b_{g,j}$ is homogeneous of degree $j$ with respect to ${\boldsymbol \l}$ if $b_{g,j}\neq0$. 
Therefore, if $j$ is odd and $b_{g,j}\neq0$, any term of $b_{g,j}$ contains a coefficient $\lambda_j^{(i)}$ such that $j$ is odd. 
Thus, we have 
\[\frac{1}{2}\frac{d}{dt}\left(1+\sum_{j=1}^{\infty}b_{g,j}t^j\right)\in\mathbb{Z}[{\boldsymbol \l}_{\mathfrak{A}}][[t]]\]
and it is homogeneous of degree $1$ with respect to ${\boldsymbol \l}_{\mathfrak{A}}$ and $t$. 
On the other hand, we have 
\[\frac{1}{1+\sum_{j=1}^{\infty}b_{g,j}t^j}=1+\sum_{\ell=1}^{\infty}\left(-\sum_{j=1}^{\infty}b_{g,j}t^j\right)^{\ell}\in\mathbb{Z}[{\boldsymbol \l}][[t]]\]
and it is homogeneous of degree $0$ with respect to ${\boldsymbol \l}$ and $t$. 
Therefore, we have 
\[\sum_{k=1}^{\infty}c_kt^{k-1}\in\mathbb{Z}[{\boldsymbol \l}_{\mathfrak{A}}][[t]]\]
and it is homogeneous of degree $1$ with respect to ${\boldsymbol \l}_{\mathfrak{A}}$ and $t$. 
Thus, we obtain the statement of the lemma. 
\end{proof}

\begin{thm}\label{maintheorem}
(i) We have $\sigma(u)\in\mathbb{Z}[{\boldsymbol \l}_{\mathfrak{A}}]\langle\langle u \rangle\rangle$. 

\vspace{1ex}

(ii) We have $\sigma(u)^2\in\mathbb{Z}[{\boldsymbol \l}]\langle\langle u \rangle\rangle$. 
\end{thm}

\begin{proof}
Since the determinant of $B$ is 1, we have 
\[B^{-1}=\widetilde{B},\]
where $\widetilde{B}$ is the adjugate matrix of $B$. 
Therefore, we have $B^{-1}\in M_g(\mathbb{Z}[{\boldsymbol \lambda}])$. 
We set $u=Bv$ in Theorem \ref{2021.9.9}. Then we have 
\begin{equation}
\sigma(u)=\exp\left(cB^{-1}u-\frac{1}{2}{}^tu\;{}^t(B^{-1})NB^{-1}u\right)\tau(B^{-1}u).\label{2021..9.9.1}
\end{equation}
From Lemma \ref{2022.3.20.888} and Proposition \ref{tauhe}, we have $\tau(B^{-1}u)\in\mathbb{Z}[{\boldsymbol \l}]\langle\langle u \rangle\rangle$. 
Let $\widetilde{c}=cB^{-1}$, $\widetilde{c}=(\widetilde{c}_1,\dots,\widetilde{c}_g)$, $\widetilde{N}={}^t(B^{-1})NB^{-1}$, and $\widetilde{N}=(\widetilde{q}_{i,j})_{1\le i,j\le g}$. 
From Lemma \ref{2021.9.10.33}, we have $\widetilde{c}_i\in\mathbb{Z}[{\boldsymbol \l}_{\mathfrak{A}}]$ for any $i$. 
Since $N$ is a symmetric matrix, $\widetilde{N}$ is also a symmetric matrix. 
From Proposition \ref{2022.3.20.1}, we have $\widetilde{q}_{i,j}\in\mathbb{Z}[{\boldsymbol \l}]$ for any $i,j$. 
Thus, we have $\exp(\widetilde{c}_iu_i)\in\mathbb{Z}[{\boldsymbol \l}_{\mathfrak{A}}]\langle\langle u_i \rangle\rangle$ and $\exp(\widetilde{q}_{i,j}u_iu_j)\in\mathbb{Z}[{\boldsymbol \l}]\langle\langle u_i, u_j\rangle\rangle$ for any $i,j$. 
For any non-negative integer $n$, we have 
\[\frac{(2n)!}{2^nn!}\in\mathbb{Z}\]
(cf. \cite[Lemma 11]{AB2020}). Therefore, we have $\exp(\widetilde{q}_{i,i}u_{i}^2/2)\in\mathbb{Z}[{\boldsymbol \l}]\langle\langle u_{i} \rangle\rangle$ for any $i$. 
Thus, from (\ref{2021..9.9.1}), we obtain the statement of (i).   
From (\ref{2021..9.9.1}), we have 
\[\sigma(u)^2=\exp\left(2cB^{-1}u-{}^tu\;{}^t(B^{-1})NB^{-1}u\right)\tau(B^{-1}u)^2.\]
From (\ref{2022.3.20.22}), we have $2c_i\in\mathbb{Z}[{\boldsymbol \l}]$ for any $i$. 
Therefore, we obtain the statement of (ii).  
\end{proof}

\begin{lem}\label{2022.5.31.111}
For $(k_1,\dots,k_m)\in2{\mathbb Z}_{\geq 0}^m$, we define $\mathfrak{c}_n$ by 
\begin{equation}
\sum_{n=0}^{\infty}\mathfrak{c}_nt^n=\left(\sum_{k=0}^{\infty}p_{1,k}t^k\right)^{k_1}\cdots\left(\sum_{k=0}^{\infty}p_{m,k}t^k\right)^{k_m},\label{2022.5.31.1}
\end{equation}
where $p_{i,k}\in\mathbb{Z}[{\boldsymbol \l}]$ is defined in (\ref{xi}). 
If $n$ is odd, then we have $\mathfrak{c}_n\in2\mathbb{Z}[{\boldsymbol \l}]$. 
\end{lem}

\begin{proof}
We differentiate the both sides of (\ref{2022.5.31.1}) with respect to $t$. 
Since $k_1,\dots,k_m$ are even non-negative integers, we have $n\mathfrak{c}_n\in2\mathbb{Z}[{\boldsymbol \l}]$ for any $n\ge1$. 
Therefore, if $n$ is odd, then we have $\mathfrak{c}_n\in2\mathbb{Z}[{\boldsymbol \l}]$. 
\end{proof}

\begin{lem}\label{2023.12.1.113}
For $(k_1,\dots,k_m), (\ell_1,\dots,\ell_m)\in{\mathbb Z}_{\geq 0}^m$ such that $\sum_{i=1}^ma_ik_i=\sum_{i=1}^ma_i\ell_i$, we have 
\begin{equation}
x_1^{k_1}\cdots x_m^{k_m}=x_1^{\ell_1}\cdots x_m^{\ell_m}+\sum\mathfrak{d}_{i_1,\dots,i_m}x_1^{i_1}\cdots x_m^{i_m},\label{2023.11.22.1}
\end{equation}
where the summation is taken over 
$(i_1,\dots, i_m)\in {\mathbb Z}_{\geq 0}^m$ such that $\sum_{j=1}^m a_ji_j<\sum_{j=1}^m a_jk_j$, $\mathfrak{d}_{i_1,\dots,i_m}\in\mathbb{Z}[{\boldsymbol \l}]$, and the right hand side of (\ref{2023.11.22.1}) is homogeneous of degree $\sum_{j=1}^ma_jk_j$ with respect to ${\boldsymbol \l}$ and $x_1,\dots,x_m$. 
\end{lem}

\begin{proof}
From Lemma \ref{2022.6.28.222}, we obtain the statement of the lemma. 
\end{proof}

When we consider a polynomial $f\in\mathbb{C}[X_1,\dots,X_m]$, we combine like terms in $f$. 
For a polynomial $f\in\mathbb{Z}[{\boldsymbol \l}][X_1,\dots,X_m]$, we consider the following two operations: 

\begin{enumerate}
\renewcommand{\labelenumi}{(\Alph{enumi}).}

\item When a term $\mathfrak{e}_1X_1^{k_1}\cdots X_m^{k_m}$ such that $\mathfrak{e}_1\in\mathbb{Z}[{\boldsymbol \l}]$ and $(k_1,\dots,k_m)\in{\mathbb Z}_{\geq 0}^m\backslash 2{\mathbb Z}_{\geq 0}^m$ appears in $f$, 
we replace $X_1^{k_1}\cdots X_m^{k_m}$ with $X_1^{\ell_1}\cdots X_m^{\ell_m}+\cdots$ in this term. 

\item When a term $\mathfrak{e}_2(X_1^{k_1}\cdots X_m^{k_m})^2$ such that $\mathfrak{e}_2\in\mathbb{Z}[{\boldsymbol \l}]$ and $(k_1,\dots,k_m)\in{\mathbb Z}_{\geq 0}^m$ appears in $f$, we replace $(X_1^{k_1}\cdots X_m^{k_m})^2$ with $(X_1^{\ell_1}\cdots X_m^{\ell_m}+\cdots)^2$ in this term. 
\end{enumerate}

In the operations (A) and (B), $X_1^{\ell_1}\cdots X_m^{\ell_m}+\cdots$ denotes the polynomial in $\mathbb{Z}[{\boldsymbol \l}][X_1,\dots,X_m]$ obtained by replacing $x_i$ of the right hand side of (\ref{2023.11.22.1}) with $X_i$ for any $1\le i\le m$.  

\begin{ex}
We consider the $(2,3)$ curve. 
We have 
\[x_1^3=x_2^2-\lambda_1x_1x_2-\lambda_2x_1^2-\lambda_3x_2-\lambda_4x_1-\lambda_6.\]
First, we consider the case of 
\[f=\lambda_1X_1^3+\lambda_1\lambda_3X_2.\]
By applying the operation (A) to $X_1^3$, the polynomial $f$ transforms into 
\begin{align*}
&\lambda_1\left(X_2^2-\lambda_1X_1X_2-\lambda_2X_1^2-\lambda_3X_2-\lambda_4X_1-\lambda_6\right)+\lambda_1\lambda_3X_2\\
&=\lambda_1X_2^2-\lambda_1^2X_1X_2-\lambda_1\lambda_2X_1^2-\lambda_1\lambda_4X_1-\lambda_1\lambda_6.
\end{align*}
Next, we consider the case of 
\[f=\lambda_1X_1^6-2\lambda_1\lambda_3\lambda_6X_2-2\lambda_1\lambda_4\lambda_6X_1-\lambda_1\lambda_6^2.\]
By applying the operation (B) to $X_1^6$, the polynomial $f$ transforms into 
\begin{align*}
&\lambda_1\left(X_2^2-\lambda_1X_1X_2-\lambda_2X_1^2-\lambda_3X_2-\lambda_4X_1-\lambda_6\right)^2-2\lambda_1\lambda_3\lambda_6X_2-2\lambda_1\lambda_4\lambda_6X_1-\lambda_1\lambda_6^2\\
&=\lambda_1X_2^4-2\lambda_1^2X_1X_2^3+\lambda_1(\lambda_1^2-2\lambda_2)X_1^2X_2^2+2\lambda_1^2\lambda_2X_1^3X_2-2\lambda_1\lambda_3X_2^3\\
&+2\lambda_1(\lambda_1\lambda_3-\lambda_4)X_1X_2^2+\lambda_1\lambda_2^2X_1^4+2\lambda_1(\lambda_1\lambda_4+\lambda_2\lambda_3)X_1^2X_2+\lambda_1(\lambda_3^2-2\lambda_6)X_2^2\\
&+2\lambda_1\lambda_2\lambda_4X_1^3+2\lambda_1(\lambda_1\lambda_6+\lambda_3\lambda_4)X_1X_2+\lambda_1(\lambda_4^2+2\lambda_2\lambda_6)X_1^2. 
\end{align*}

\end{ex}

For a subring $R$ of $\mathbb{C}$, let $\mathcal{P}(R)$ be the set of $\sum \mathfrak{f}_{i_1,\dots,i_m}X_1^{i_1}\cdots X_m^{i_m}\in R[X_1,\dots,X_m]$ 
such that $\mathfrak{f}_{i_1,\dots,i_m}\in 2R$ or $(i_1,\dots,i_m)\in2{\mathbb Z}_{\geq 0}^m$.  
For $f\in\mathbb{Z}[{\boldsymbol \l}][X_1,\dots,X_m]$, 
let $\mathcal{S}(f)$ be the set of the subsets $\mathfrak{S}$ of ${\boldsymbol \l}$ such that 
$f\in\mathcal{P}(\mathbb{Z}[{\boldsymbol \l}_{\mathfrak{S}}])$. 
If $\mathfrak{S}_0\in\mathcal{S}(f)$, then we have $\mathfrak{S}\in\mathcal{S}(f)$ for any subset $\mathfrak{S}$ of ${\boldsymbol \l}$ such that $\mathfrak{S}_0\subseteq\mathfrak{S}$. 

\begin{ex}
We consider the $(2,3)$ curve. 
First, we consider the case of 
\[f=X_1^2X_2^2+2X_1^5+\lambda_1X_2^3+\lambda_4X_1^3+2\lambda_3X_1^2X_2+\lambda_6X_1^2+\lambda_2^4X_1.\]
The set $\mathcal{S}(f)$ is the set of the subsets of ${\boldsymbol \l}$ which contain $\lambda_1, \lambda_2, \lambda_4$. 
We have $\mathcal{S}(f)=\{\mathfrak{S}_1, \mathfrak{S}_2, \mathfrak{S}_3, \mathfrak{S}_4\}$, where 
\begin{align*}
\mathfrak{S}_1&=\{\lambda_1, \lambda_2, \lambda_4\},\qquad \mathfrak{S}_2=\{\lambda_1, \lambda_2, \lambda_3, \lambda_4\},\\
\mathfrak{S}_3&=\{\lambda_1, \lambda_2, \lambda_4, \lambda_6\},\qquad \mathfrak{S}_4=\{\lambda_1, \lambda_2, \lambda_3, \lambda_4, \lambda_6\}. 
\end{align*}
Next, we consider the case of 
\[f=X_1^5+\lambda_1X_2^3+\lambda_4X_1^3+2\lambda_3X_1^2X_2+\lambda_6X_1^2+\lambda_2^4X_1.\]
We have $\mathcal{S}(f)=\emptyset$. 
\end{ex}

For $f_1, f_2\in\mathbb{Z}[{\boldsymbol \l}][X_1,\dots,X_m]$ and $\mathfrak{S}\subseteq{\boldsymbol \l}$, if  $\mathfrak{S}\in\mathcal{S}(f_1)\cap\mathcal{S}(f_2)$, then we have 
$\mathfrak{S}\in\mathcal{S}(f_1+f_2)$. 

\begin{lem}\label{2023.12.31.1111}
Let $f_1$ and $f_2$ be polynomials in $\mathbb{Z}[{\boldsymbol \l}][X_1,\dots,X_m]$ such that $f_1-f_2\in\mathcal{P}(\mathbb{Z}[{\boldsymbol \l}])$. 
Then we have $\mathcal{S}(f_1)=\mathcal{S}(f_2)$. 
\end{lem}

\begin{proof}
There is a polynomial $f_3\in\mathcal{P}(\mathbb{Z}[{\boldsymbol \l}])$ such that $f_1=f_2+f_3$. 
We assume $\mathfrak{S}\in \mathcal{S}(f_2)$. 
From $\mathfrak{S}\in \mathcal{S}(f_3)$, we have $\mathfrak{S}\in \mathcal{S}(f_1)$. 
Thus, we have $\mathcal{S}(f_2)\subseteq\mathcal{S}(f_1)$. 
In the same way, we have $\mathcal{S}(f_1)\subseteq\mathcal{S}(f_2)$. 
Therefore, we have $\mathcal{S}(f_1)=\mathcal{S}(f_2)$. 
\end{proof}

For $f_1, f_2\in\mathbb{C}[X_1,\dots,X_m]$, if $f_1-f_2\in I$, we say that $f_1$ is congruent to $f_2$ modulo $I$, where $I$ is defined in Section \ref{2023.11.26.752}. 

\begin{lem}\label{2023.11.23.456}
For $f\in\mathbb{Z}[{\boldsymbol \l}][X_1,\dots,X_m]$, 
let $f_A$ and $f_B$ be polynomials obtained by applying the operations (A) and (B) to $f$, respectively. 
Then $f$ is congruent to $f_A$ and $f_B$ modulo $I$ and we have $\mathcal{S}(f)\subseteq\mathcal{S}(f_A)$ and $\mathcal{S}(f)=\mathcal{S}(f_B)$. 
\end{lem}

\begin{proof}
From Lemma \ref{2023.12.1.113}, $f$ is congruent to $f_A$ and $f_B$ modulo $I$. 
Let $f=\sum\mathfrak{g}_{i_1,\dots,i_m}X_1^{i_1}\cdots X_m^{i_m}$. 
We assume $\mathfrak{S}\in\mathcal{S}(f)$. 
For any term $\mathfrak{g}_{i_1,\dots,i_m}X_1^{i_1}\cdots X_m^{i_m}$ in $f$ such that $(i_1,\dots,i_m)\notin2{\mathbb Z}_{\geq 0}^m$, we have 
$\mathfrak{g}_{i_1,\dots,i_m}\in 2\mathbb{Z}[{\boldsymbol \l}_{\mathfrak{S}}]$. 
Thus, we have $\mathfrak{S}\in\mathcal{S}(f_A)$. 
Therefore, we have $\mathcal{S}(f)\subseteq\mathcal{S}(f_A)$. 
In the operation (B), we have 
\[(X_1^{\ell_1}\cdots X_m^{\ell_m}+\cdots)^2\in \mathcal{P}(\mathbb{Z}[{\boldsymbol \l}]).\] 
From Lemma \ref{2023.12.31.1111}, we have $\mathcal{S}(f)=\mathcal{S}(f_B)$. 
\end{proof}

\begin{lem}\label{2023.10.22.2222}
Let $k$ be an integer such that $1\le k\le m$ and $a_k$ is odd. 
Let $f$ be a polynomial in $\mathbb{Z}[{\boldsymbol \l}][X_1,\dots,X_m]$ such that $f$ is congruent to $\det G_k$ modulo $I$. 
If $\mathfrak{S}\in\mathcal{S}(f)$, then we have $\sigma(u)\in\mathbb{Z}[{\boldsymbol \l}_{\mathfrak{S}}]\langle\langle u \rangle\rangle$. 
\end{lem}

\begin{proof}
From Propositions \ref{prop11}, \ref{prop12}, and Lemma \ref{gkp}, for $P=(x_1,\dots,x_m)\in V$, $\det G_k(P)$ can be expanded around $\infty$ with respect to $t$ as follows: 
\[\det G_k(P)=\frac{(-1)^{k+1}a_k}{t^{2g-1+a_k}}\left(1+\sum_{n=1}^{\infty}\mathfrak{h}_nt^n\right),\]
where $\mathfrak{h}_n\in\mathbb{Z}[1/a_k, {\boldsymbol \l}_{\mathfrak{S}}]$ for any $n\ge1$. 
We have $\det G_k(P)=f(P)$. 
From $f\in\mathcal{P}(\mathbb{Z}[{\boldsymbol \l}_{\mathfrak{S}}])$ and Lemma \ref{2022.5.31.111}, if $n$ is odd, then we have $\mathfrak{h}_n\in2\mathbb{Z}[1/a_k, {\boldsymbol \l}_{\mathfrak{S}}]$. 
From 
\[\frac{1}{\det G_k(P)}=\frac{t^{2g-1+a_k}}{(-1)^{k+1}a_k}\left\{1+\sum_{i=1}^{\infty}\left(-\sum_{n=1}^{\infty}\mathfrak{h}_nt^n\right)^i\right\},\]
we have the expansion 
\begin{equation}
\frac{1}{\det G_k(P)}=\frac{t^{2g-1+a_k}}{(-1)^{k+1}a_k}\left(1+\sum_{n=1}^{\infty}\mathfrak{i}_nt^n\right),\label{2022.6.2.1}
\end{equation}
where $\mathfrak{i}_n\in\mathbb{Z}[1/a_k, {\boldsymbol \l}_{\mathfrak{S}}]$ for any $n\ge1$. 
Further, if $n$ is odd, then we have $\mathfrak{i}_n\in2\mathbb{Z}[1/a_k, {\boldsymbol \l}_{\mathfrak{S}}]$. 
On the other hand, we have the expansion around $\infty$ with respect to $t$
\begin{equation}
dx_k=\frac{-a_k}{t^{a_k+1}}\left(1+\sum_{n=1}^{\infty}\mathfrak{j}_nt^n\right),\label{2022.6.2.2}
\end{equation}
where $\mathfrak{j}_n\in\mathbb{Z}[1/a_k, {\boldsymbol \l}_{\mathfrak{S}}]$ for any $n\ge1$. 
Further, if $n$ is odd, then we have $\mathfrak{j}_n\in2\mathbb{Z}[1/a_k, {\boldsymbol \l}_{\mathfrak{S}}]$. 
From Proposition \ref{2021.8.8.1}, we have 
$b_{g,j}\in\mathbb{Z}[{\boldsymbol \l}_{\mathfrak{S}}]$ for any $j\ge1$. 
From (\ref{2022.6.2.3}), (\ref{2022.6.2.1}), and (\ref{2022.6.2.2}), if $j$ is odd, then we have $b_{g,j}\in2\mathbb{Z}[{\boldsymbol \l}_{\mathfrak{S}}]$. 
Thus, we have 
\[\frac{1}{2}\frac{d}{dt}\left(1+\sum_{j=1}^{\infty}b_{g,j}t^j\right)\in\mathbb{Z}[{\boldsymbol \l}_{\mathfrak{S}}][[t]].\]
As in the case of Lemma \ref{2021.9.10.33}, we have $c_k\in\mathbb{Z}[{\boldsymbol \l}_{\mathfrak{S}}]$ for any $k\ge1$. 
Therefore, as in the case of Theorem \ref{maintheorem} (i), we obtain the statement of the lemma.  
\end{proof}

In Lemma \ref{2023.10.22.2222}, we have the following problems: 

\begin{itemize}
\item For any telescopic curve and any integer $k$ such that $1\le k\le m$ and $a_k$ is odd, is there a polynomial $f$ in $\mathbb{Z}[{\boldsymbol \l}][X_1,\dots,X_m]$ such that $f$ is congruent to $\det G_k$ modulo $I$ and $\mathcal{S}(f)\neq\emptyset$?

\item Find a polynomial $\mathscr{F}$ in $\mathbb{Z}[{\boldsymbol \l}][X_1,\dots,X_m]$ such that $\mathscr{F}$ gives the best result among all the polynomials in $\mathbb{Z}[{\boldsymbol \l}][X_1,\dots,X_m]$ which are congruent to $\det G_k$ modulo $I$. 

\end{itemize}

We solve these problems in Theorem \ref{2023.11.26.14598}. 

\begin{thm}\label{2023.11.26.14598}
Let $k$ be an integer such that $1\le k\le m$ and $a_k$ is odd. 

\vspace{1ex}

\noindent (i) By applying the operations (A) and (B), we can transform $\det G_k$ into 
\begin{equation}
\mathscr{F}=(-1)^{k+1}a_kX_1^{2h_1}\cdots X_m^{2h_m}+\sum \mathfrak{p}_{i_1,\dots,i_m}X_1^{i_1}\cdots X_m^{i_m},\label{2023.12.3.1}
\end{equation}
where
\begin{itemize}

\item $(h_1,\dots,h_m)\in {\mathbb Z}_{\geq 0}^m$ such that $2\sum_{j=1}^ma_jh_j=2g-1+a_k$,

\item the summation in (\ref{2023.12.3.1}) is taken over $(i_1,\dots,i_m)\in {\mathbb Z}_{\geq 0}^m$ such that $\sum_{j=1}^ma_ji_j<2g-1+a_k$, 

\item $\mathfrak{p}_{i_1,\dots,i_m}\in\mathbb{Z}[{\boldsymbol \l}]$, 

\item $\mathscr{F}$ is homogeneous of degree $2g-1+a_k$ with respect to ${\boldsymbol \l}$ and $X_1,\dots,X_m$, 

\item if $(i_1,\dots,i_m)\neq(i_1',\dots,i_m')$ and $\mathfrak{p}_{i_1,\dots,i_m}\mathfrak{p}_{i_1',\dots,i_m'}\neq0$, then $\sum_{j=1}^ma_ji_j\neq\sum_{j=1}^ma_ji_j'$, 

\item if $(i_1,\dots,i_m)\notin2{\mathbb Z}_{\geq 0}^m$ and $\mathfrak{p}_{i_1,\dots,i_m}\neq0$, then $\sum_{j=1}^ma_ji_j\notin 2\langle A_m\rangle$. 
\end{itemize}

\vspace{1ex}

\noindent (ii) We have $\mathcal{S}(\mathscr{F})\neq\emptyset$. For any $\mathfrak{S}\in \mathcal{S}(\mathscr{F})$, we have $\sigma(u)\in\mathbb{Z}[{\boldsymbol \l}_{\mathfrak{S}}]\langle\langle u \rangle\rangle$.

\vspace{1ex}

\noindent (iii) For any $f\in\mathbb{Z}[{\boldsymbol \l}][X_1,\dots,X_m]$ which is congruent to $\det G_k$ modulo $I$, we have $\mathcal{S}(f)\subseteq\mathcal{S}(\mathscr{F})$. 
\end{thm}

\begin{proof}
(i) We have $\det G_k\in\mathbb{Z}[{\boldsymbol \l}][X_1,\dots,X_m]$ and $\det G_k$ is homogeneous of degree $2g-1+a_k$ with respect to ${\boldsymbol \l}$ and $X_1,\dots,X_m$. 
From Lemma \ref{gkp}, the maximum of the degree of $X_1^{i_1}\cdots X_m^{i_m}$ appearing in $\det G_k$ is $2g-1+a_k$. 
From Lemma \ref{2023.10.26.1}, we have $2g-1+a_k\in2\langle A_m\rangle$. 
From Lemma \ref{gkp}, by applying the operations (A) and (B), we can transform $\det G_k$ into 
\[(-1)^{k+1}a_kX_1^{2h_1}\cdots X_m^{2h_m}+\mathscr{G},\]
where $(h_1,\dots,h_m)\in {\mathbb Z}_{\geq 0}^m$ such that $2\sum_{j=1}^ma_jh_j=2g-1+a_k$, $\mathscr{G}\in\mathbb{Z}[{\boldsymbol \l}][X_1,\dots,X_m]$, and the maximum of the degree of $X_1^{i_1}\cdots X_m^{i_m}$ appearing in $\mathscr{G}$ is less than $2g-1+a_k$. 
Let $\mathfrak{k}$ be the maximum of the degree of $X_1^{i_1}\cdots X_m^{i_m}$ appearing in $\mathscr{G}$. 
\begin{itemize}
\item If $\mathfrak{k}\in2\langle A_m\rangle$, we take a sequence $(h_1^{(1)},\dots,h_m^{(1)})\in {\mathbb Z}_{\geq 0}^m$ such that $2\sum_{j=1}^ma_jh_j^{(1)}=\mathfrak{k}$. 
By applying the operations (A) and (B), we replace all $X_1^{i_1}\cdots X_m^{i_m}$ in $\mathscr{G}$ such that $\sum_{j=1}^ma_ji_j=\mathfrak{k}$ 
with $X_1^{2h_1^{(1)}}\cdots X_m^{2h_m^{(1)}}+\cdots$. 

\item If $\mathfrak{k}\notin2\langle A_m\rangle$, we take a sequence $(h_1^{(2)},\dots,h_m^{(2)})\in {\mathbb Z}_{\geq 0}^m$ such that $\sum_{j=1}^ma_jh_j^{(2)}=\mathfrak{k}$. 
By applying the operation (A), we replace all $X_1^{i_1}\cdots X_m^{i_m}$ in $\mathscr{G}$ such that $\sum_{j=1}^ma_ji_j=\mathfrak{k}$ with $X_1^{h_1^{(2)}}\cdots X_m^{h_m^{(2)}}+\cdots$. 
\end{itemize}
By repeating the above operations, we obtain the statement of (i). 

\vspace{1ex}

(ii) Since any $\mathfrak{p}_{i_1,\dots,i_m}$ sucu that $\mathfrak{p}_{i_1,\dots,i_m}\neq0$ contains a coefficient $\lambda_j^{(i)}$ in ${\boldsymbol \lambda}$, we have $\mathcal{S}(\mathscr{F})\neq\emptyset$. 
From Lemma \ref{2023.11.23.456}, $\mathscr{F}$ is congruent to $\det G_k$ modulo $I$. 
From Lemma \ref{2023.10.22.2222}, for any $\mathfrak{S}\in \mathcal{S}(\mathscr{F})$, we have $\sigma(u)\in\mathbb{Z}[{\boldsymbol \l}_{\mathfrak{S}}]\langle\langle u \rangle\rangle$. 

\vspace{1ex}

(iii) Let $f$ be a polynomial in $\mathbb{Z}[{\boldsymbol \l}][X_1,\dots,X_m]$ which is congruent to $\det G_k$ modulo $I$. 
Let $\mathfrak{l}$ be the maximum of the degree of $X_1^{i_1}\cdots X_m^{i_m}$ appearing in $f$. 
Then we have $\mathfrak{l}\ge2g-1+a_k$. 
We consider the case of $\mathfrak{l}>2g-1+a_k$. 
\begin{itemize}
\item If $\mathfrak{l}\in2\langle A_m\rangle$, we take a sequence $(h_1^{(3)},\dots,h_m^{(3)})\in {\mathbb Z}_{\geq 0}^m$ such that $2\sum_{j=1}^ma_jh_j^{(3)}=\mathfrak{l}$. 
By applying the operations (A) and (B), we replace all $X_1^{i_1}\cdots X_m^{i_m}$ in $f$ such that $\sum_{j=1}^ma_ji_j=\mathfrak{l}$ 
with $X_1^{2h_1^{(3)}}\cdots X_m^{2h_m^{(3)}}+\cdots$. 
\item If $\mathfrak{l}\notin2\langle A_m\rangle$, we take a sequence $(h_1^{(4)},\dots,h_m^{(4)})\in {\mathbb Z}_{\geq 0}^m$ such that $\sum_{j=1}^ma_jh_j^{(4)}=\mathfrak{l}$. 
By applying the operation (A), we replace all  
$X_1^{i_1}\cdots X_m^{i_m}$ in $f$ such that $\sum_{j=1}^ma_ji_j=\mathfrak{l}$ with $X_1^{h_1^{(4)}}\cdots X_m^{h_m^{(4)}}+\cdots$. 
\end{itemize}
By repeating the above operations, we can transform $f$ into a polynomial $\mathscr{H}$ in $\mathbb{Z}[{\boldsymbol \l}][X_1,\dots,X_m]$, where the maximum of the degree of $X_1^{i_1}\cdots X_m^{i_m}$ appearing in $\mathscr{H}$ is $2g-1+a_k$. 
By applying the operations (A) and (B), we can transform $\mathscr{H}$ or $f$ in the case of $\mathfrak{l}=2g-1+a_k$ into 
\[(-1)^{k+1}a_kX_1^{2h_1}\cdots X_m^{2h_m}+\mathscr{I},\]
where $(h_1,\dots,h_m)$ is defined in (\ref{2023.12.3.1}), $\mathscr{I}\in\mathbb{Z}[{\boldsymbol \l}][X_1,\dots,X_m]$, and the maximum of the degree of $X_1^{i_1}\cdots X_m^{i_m}$ appearing in $\mathscr{I}$ is less than $2g-1+a_k$. 
Let $\mathfrak{m}$ be the maximum of the degree of $X_1^{i_1}\cdots X_m^{i_m}$ appearing in $\mathscr{I}$. 
\begin{itemize}
\item If we have the term $\mathfrak{p}_{k_1,\dots,k_m}X_1^{k_1}\cdots X_m^{k_m}$ such that $\mathfrak{p}_{k_1,\dots,k_m}\neq0$ and $\sum_{j=1}^ma_jk_j=\mathfrak{m}$ in (\ref{2023.12.3.1}), 
by applying the operations (A) and (B), we replace all $X_1^{i_1}\cdots X_m^{i_m}$ in $\mathscr{I}$ such that $\sum_{j=1}^ma_ji_j=\mathfrak{m}$ 
with $X_1^{k_1}\cdots X_m^{k_m}+\cdots$. 

\item If $\mathfrak{m}\in2\langle A_m\rangle$ and there is no term $\mathfrak{p}_{k_1,\dots,k_m}X_1^{k_1}\cdots X_m^{k_m}$ such that $\mathfrak{p}_{k_1,\dots,k_m}\neq0$ and $\sum_{j=1}^ma_jk_j=\mathfrak{m}$ in (\ref{2023.12.3.1}), 
we take a sequence $(h_1^{(5)},\dots,h_m^{(5)})\in {\mathbb Z}_{\geq 0}^m$ such that $2\sum_{j=1}^ma_jh_j^{(5)}=\mathfrak{m}$. 
By applying the operations (A) and (B), we replace all  
$X_1^{i_1}\cdots X_m^{i_m}$ in $\mathscr{I}$ such that $\sum_{j=1}^ma_ji_j=\mathfrak{m}$ with $X_1^{2h_1^{(5)}}\cdots X_m^{2h_m^{(5)}}+\cdots$. 

\item If $\mathfrak{m}\notin2\langle A_m\rangle$ and there is no term $\mathfrak{p}_{k_1,\dots,k_m}X_1^{k_1}\cdots X_m^{k_m}$ such that $\mathfrak{p}_{k_1,\dots,k_m}\neq0$ and $\sum_{j=1}^ma_jk_j=\mathfrak{m}$ in (\ref{2023.12.3.1}), 
we take a sequence $(h_1^{(6)},\dots,h_m^{(6)})\in {\mathbb Z}_{\geq 0}^m$ such that $\sum_{j=1}^ma_jh_j^{(6)}=\mathfrak{m}$. 
By applying the operation (A), we replace all $X_1^{i_1}\cdots X_m^{i_m}$ in $\mathscr{I}$ such that $\sum_{j=1}^ma_ji_j=\mathfrak{m}$ with $X_1^{h_1^{(6)}}\cdots X_m^{h_m^{(6)}}+\cdots$. 
\end{itemize}
By repeating the above operations, we can transform $f$ into $\mathscr{F}$. 
From Lemma \ref{2023.11.23.456}, we have $\mathcal{S}(f)\subseteq\mathcal{S}(\mathscr{F})$. 
\end{proof}

For a non-negative integer $n$, we define $\chi(n)$ by 
\[\chi(n)=\left\{\begin{array}{ll}0 & \quad \mbox{if $n$ is even} \\ 1 & \quad \mbox{if $n$ is odd}\end{array}\right..\]
Let $\mathfrak{B}$ be the set of $\lambda_j^{(i)}$ which is the coefficient of $X_1^{j_1}\cdots X_m^{j_m}$ with $\sum_{k=1}^m\chi(j_k)\ge2$ in (\ref{eq-2-5}). 

\begin{prop}\label{2022.5.29.1}
If $\sum_{j=1}^{i-1}\chi(\ell_{i,j})\le1$ for any $2\le i\le m$, where $\ell_{i,j}$ is defined in (\ref{homo}), then we have $\sigma(u)\in\mathbb{Z}[{\boldsymbol \l}_{\mathfrak{B}}]\langle\langle u \rangle\rangle$. 
\end{prop}

\begin{proof}
We take an integer $k$ such that $1\le k\le m$ and $a_k$ is odd. 
From $\sum_{j=1}^{i-1}\chi(\ell_{i,j})\le1$, for any $2\le i\le m$ and $1\le j\le m$, we have $\frac{\partial F_i}{\partial X_j}\in\mathcal{P}(\mathbb{Z}[{\boldsymbol \l}_{\mathfrak{B}}])$. 
Thus, we have $\det G_k\in\mathcal{P}(\mathbb{Z}[{\boldsymbol \l}_{\mathfrak{B}}])$. Therefore, we have $\mathfrak{B}\in\mathcal{S}(\det G_k)$. 
From Lemma \ref{2023.10.22.2222}, we obtain the statement of the proposition. 
\end{proof}

\begin{rem}
We can apply Theorems \ref{maintheorem} and \ref{2023.11.26.14598} to any telescopic curve. 
In the case that the condition of Proposition \ref{2022.5.29.1} holds, Theorem \ref{2023.11.26.14598} gives the better result than Proposition \ref{2022.5.29.1}. 
On the other hand, in Proposition \ref{2022.5.29.1}, we do not need to calculate $\det G_k$. 
\end{rem}

\begin{lem}\label{2023.12.10.1}
We consider the case of $m=2$.  
For $(i_1,i_2), (j_1,j_2)\in{\mathbb Z}_{\geq 0}^2$, if $a_1i_1+a_2i_2=a_1j_1+a_2j_2<a_1a_2$, then we have 
$(i_1,i_2)=(j_1,j_2)$. 
\end{lem}

\begin{proof}
From $a_1i_1+a_2i_2=a_1j_1+a_2j_2<a_1a_2$, we have $i_2, j_2<a_1$. 
Thus, we have $(i_1,i_2), (j_1,j_2)\in B(A_2)$. 
From Lemma \ref{lem-2-1}, we have $(i_1,i_2)=(j_1,j_2)$. 
\end{proof}

\begin{rem}
We consider the $(n,s)$ curve. 
The polynomial $F_2$ defining the $(n,s)$ curve is given in Example \ref{2022.5.29.2} (i). 
Let $\mathfrak{C}$ be the set of $\lambda_j$ which is the coefficient of $X_1^{j_1}X_2^{j_2}$ with $(j_1,j_2)=(\mbox{odd}, \mbox{odd})$ in (\ref{2023.12.9.1}).  
In \cite[Theorem 2.3]{O-2018}, it is proved that we have $\sigma(u)\in\mathbb{Z}[{\boldsymbol \l}_{\mathfrak{C}}]\langle\langle u \rangle\rangle$.  
\end{rem}

\begin{rem}\label{2023.12.19.1}
We apply Theorem \ref{maintheorem} (i) to the $(n,s)$ curve. 
In the case of $(n,s)=(\mbox{odd}, \mbox{odd})$, $\mathfrak{A}$ is the set of $\lambda_j$ which is the coefficient of $X_1^{j_1}X_2^{j_2}$ with $(j_1,j_2)=(\mbox{odd}, \mbox{odd})$ or $(\mbox{even}, \mbox{even})$. 
In the case of $(n,s)=(\mbox{odd}, \mbox{even})$, $\mathfrak{A}$ is the set of $\lambda_j$ which is the coefficient of $X_1^{j_1}X_2^{j_2}$ with $(j_1,j_2)=(\mbox{odd}, \mbox{odd})$ or $(\mbox{odd}, \mbox{even})$. 
In the case of $(n,s)=(\mbox{even}, \mbox{odd})$, $\mathfrak{A}$ is the set of $\lambda_j$ which is the coefficient of $X_1^{j_1}X_2^{j_2}$ with $(j_1,j_2)=(\mbox{odd}, \mbox{odd})$ or $(\mbox{even}, \mbox{odd})$. 
Therefore, \cite[Theorem 2.3]{O-2018} gives the better result than Theorem \ref{maintheorem} (i) for $(n,s)$ curves. 
\end{rem}

\begin{rem}\label{2023.12.19.2}
We apply Theorem \ref{2023.11.26.14598} to the $(n,s)$ curve. 
First, we consider the case that $n$ is odd. 
We have 
\begin{equation}
\det G_1=nX_2^{n-1}-\sum j_2\lambda_{ns-nj_1-sj_2}X_1^{j_1}X_2^{j_2-1},\label{2023.12.9.22}
\end{equation}
where the summation is taken over $(j_1,j_2)\in{\mathbb Z}_{\geq 0}^2$ such that $nj_1+sj_2<ns$ and $j_2\ge1$. 
From Lemma \ref{2023.12.10.1}, we find that the polynomial (\ref{2023.12.9.22}) is in the form of (\ref{2023.12.3.1}). 
We can express the polynomial (\ref{2023.12.9.22}) in the form of $\mathscr{J}_1+\mathscr{J}_2$ with $\mathscr{J}_2\in\mathcal{P}(\mathbb{Z}[{\boldsymbol \l}])$ and 
\[\mathscr{J}_1=\sum\lambda_{ns-nj_1-sj_2}X_1^{j_1}X_2^{j_2-1},\]
where the summation is taken over $(j_1,j_2)\in{\mathbb Z}_{\geq 0}^2$ such that $(j_1,j_2)=(\mbox{odd}, \mbox{odd})$ and $nj_1+sj_2<ns$. 
From Lemma \ref{2023.12.31.1111}, we have $\mathcal{S}(\det G_1)=\mathcal{S}(\mathscr{J}_1)$. 
Thus, we obtain the same result as \cite[Theorem 2.3]{O-2018}. 
Next, we consider the case that $s$ is odd. 
We have 
\begin{equation}
\det G_2=-sX_1^{s-1}-\sum j_1\lambda_{ns-nj_1-sj_2}X_1^{j_1-1}X_2^{j_2},\label{2023.12.10.953}
\end{equation}
where the summation is taken over $(j_1,j_2)\in{\mathbb Z}_{\geq 0}^2$ such that $nj_1+sj_2<ns$ and $j_1\ge1$. 
From Lemma \ref{2023.12.10.1}, we find that the polynomial (\ref{2023.12.10.953}) is in the form of (\ref{2023.12.3.1}). 
We can express the polynomial (\ref{2023.12.10.953}) in the form of $\mathscr{K}_1+\mathscr{K}_2$ with $\mathscr{K}_2\in\mathcal{P}(\mathbb{Z}[{\boldsymbol \l}])$ and 
\[\mathscr{K}_1=\sum\lambda_{ns-nj_1-sj_2}X_1^{j_1-1}X_2^{j_2},\]
where the summation is taken over $(j_1,j_2)\in{\mathbb Z}_{\geq 0}^2$ such that $(j_1,j_2)=(\mbox{odd}, \mbox{odd})$ and $nj_1+sj_2<ns$. 
From Lemma \ref{2023.12.31.1111}, we have $\mathcal{S}(\det G_2)=\mathcal{S}(\mathscr{K}_1)$. 
Thus, we obtain the same result as \cite[Theorem 2.3]{O-2018}. 
Therefore, Theorem \ref{2023.11.26.14598} includes \cite[Theorem 2.3]{O-2018}. 
\end{rem}

\begin{rem}
We can apply Proposition \ref{2022.5.29.1} to the $(n,s)$ curve. 
Applying Proposition \ref{2022.5.29.1} to the $(n,s)$ curve, we have \cite[Theorem 2.3]{O-2018} immediately. 
We consider the case of $m=3$. 
The sequence $A_3$ can be uniquely expressed as $A_3=(a_1'd, a_2'd, k_1a_1'+k_2a_2')$, where $a_1', a_2', d\in\mathbb{N}$ and $k_1,k_2\in{\mathbb Z}_{\geq 0}$ such that 
$\mathrm{gcd}(a_1',a_2')=\mathrm{gcd}(d,k_1a_1'+k_2a_2')=1$, $a_1'd, a_2'd, k_1a_1'+k_2a_2'\geq2$, and $k_2<a_1'$. 
Then we have $\ell_{3,1}=k_1$ and $\ell_{3,2}=k_2$. 
If $k_1$ or $k_2$ is even, then $A_3$ satisfies the condition of Proposition \ref{2022.5.29.1}. 
For example, for any non-negative integer $k$, $A_3=(4,6,4k+3)$ satisfies the condition of Proposition \ref{2022.5.29.1}. 
On the other hand, if both $k_1$ and $k_2$ are odd, then $A_3$ does not satisfy the condition of Proposition \ref{2022.5.29.1}. 
For example, for any positive integer $k$, $A_3=(4,6,4k+1)$ does not satisfy the condition of Proposition \ref{2022.5.29.1}. 
We can apply Proposition \ref{2022.5.29.1} to the telescopic curve considered in Example \ref{2022.5.29.2} (v). 
\end{rem}

\begin{ex}\label{2023.12.19.3}
We consider the case of $m=3$ and $A_3=(4,6,5)$. 
The polynomials $F_2$ and $F_3$ defining this curve are given in Example \ref{2022.5.29.2} (ii). 
We have $\mathfrak{A}=\{\lambda_1^{(2)}, \lambda_3^{(2)}, \lambda_7^{(2)}, \lambda_1^{(3)}, \lambda_5^{(3)}\}$. 
From Theorem \ref{maintheorem} (i), we have $\sigma(u)\in\mathbb{Z}[{\boldsymbol \l}_{\mathfrak{A}}]\langle\langle u \rangle\rangle$.  
By applying the operation (A), we can transform $\det G_3$ into
\begin{align*}
\mathscr{L}&=5X_2^2+2(\lambda_1^{(3)}-2\lambda_1^{(2)})X_2X_3+(4\lambda_2^{(3)}-\lambda_1^{(2)}\lambda_1^{(3)}-3\lambda_2^{(2)})X_3^2\\
&+2(\lambda_2^{(2)}\lambda_1^{(3)}-\lambda_1^{(2)}\lambda_2^{(3)}-\lambda_3^{(2)}-2\lambda_1^{(3)}\lambda_2^{(3)})X_1X_3\\
&+\left(\lambda_2^{(2)}\lambda_2^{(3)}-\lambda_4^{(2)}+3\lambda_4^{(3)}-4(\lambda_2^{(3)})^2\right)X_1^2\\
&+2(\lambda_6^{(3)}+2\lambda_2^{(2)}\lambda_4^{(3)}-2\lambda_6^{(2)}-2\lambda_2^{(3)}\lambda_4^{(3)})X_2\\
&+(-\lambda_1^{(2)}\lambda_6^{(3)}+3\lambda_2^{(2)}\lambda_5^{(3)}+\lambda_3^{(2)}\lambda_4^{(3)}-\lambda_6^{(2)}\lambda_1^{(3)}-3\lambda_7^{(2)}-4\lambda_2^{(3)}\lambda_5^{(3)})X_3\\
&+2(\lambda_2^{(2)}\lambda_6^{(3)}+\lambda_4^{(2)}\lambda_4^{(3)}-\lambda_6^{(2)}\lambda_2^{(3)}-\lambda_8^{(2)}-2\lambda_2^{(3)}\lambda_6^{(3)})X_1\\
&+3\lambda_2^{(2)}\lambda_{10}^{(3)}-\lambda_6^{(2)}\lambda_6^{(3)}+\lambda_8^{(2)}\lambda_4^{(3)}-3\lambda_{12}^{(2)}-4\lambda_2^{(3)}\lambda_{10}^{(3)}.
\end{align*}
The polynomial $\mathscr{L}$ is in the form of (\ref{2023.12.3.1}). 
We have $\mathscr{L}=\mathscr{L}_1+\mathscr{L}_2$ with $\mathscr{L}_2\in\mathcal{P}(\mathbb{Z}[{\boldsymbol \l}])$ and 
\[\mathscr{L}_1=(\lambda_1^{(2)}\lambda_6^{(3)}+\lambda_2^{(2)}\lambda_5^{(3)}+\lambda_3^{(2)}\lambda_4^{(3)}+\lambda_6^{(2)}\lambda_1^{(3)}+\lambda_7^{(2)})X_3.\]
From Lemma \ref{2023.12.31.1111}, we have $\mathcal{S}(\mathscr{L})=\mathcal{S}(\mathscr{L}_1)$. 
Let 
\begin{align*}
\mathfrak{D}_1&=\{\lambda_1^{(2)}, \lambda_2^{(2)}, \lambda_3^{(2)}, \lambda_6^{(2)}, \lambda_7^{(2)}\},\qquad \mathfrak{D}_2=\{\lambda_1^{(2)}, \lambda_2^{(2)}, \lambda_3^{(2)}, \lambda_7^{(2)}, \lambda_1^{(3)}\},\\
\mathfrak{D}_3&=\{\lambda_1^{(2)}, \lambda_2^{(2)}, \lambda_6^{(2)}, \lambda_7^{(2)}, \lambda_4^{(3)}\},\qquad \mathfrak{D}_4=\{\lambda_1^{(2)}, \lambda_2^{(2)}, \lambda_7^{(2)}, \lambda_1^{(3)}, \lambda_4^{(3)}\},\\
\mathfrak{D}_5&=\{\lambda_1^{(2)}, \lambda_3^{(2)}, \lambda_6^{(2)}, \lambda_7^{(2)}, \lambda_5^{(3)}\},\qquad \mathfrak{D}_6=\{\lambda_1^{(2)}, \lambda_3^{(2)}, \lambda_7^{(2)}, \lambda_1^{(3)}, \lambda_5^{(3)}\},\\
\mathfrak{D}_7&=\{\lambda_1^{(2)}, \lambda_6^{(2)}, \lambda_7^{(2)}, \lambda_4^{(3)}, \lambda_5^{(3)}\},\qquad \mathfrak{D}_8=\{\lambda_1^{(2)}, \lambda_7^{(2)}, \lambda_1^{(3)}, \lambda_4^{(3)}, \lambda_5^{(3)}\},\\
\mathfrak{D}_9&=\{\lambda_2^{(2)}, \lambda_3^{(2)}, \lambda_6^{(2)}, \lambda_7^{(2)}, \lambda_6^{(3)}\},\qquad \mathfrak{D}_{10}=\{\lambda_2^{(2)}, \lambda_3^{(2)}, \lambda_7^{(2)}, \lambda_1^{(3)}, \lambda_6^{(3)}\},\\
\mathfrak{D}_{11}&=\{\lambda_2^{(2)}, \lambda_6^{(2)}, \lambda_7^{(2)}, \lambda_4^{(3)}, \lambda_6^{(3)}\},\qquad \mathfrak{D}_{12}=\{\lambda_2^{(2)}, \lambda_7^{(2)}, \lambda_1^{(3)}, \lambda_4^{(3)}, \lambda_6^{(3)}\},\\
\mathfrak{D}_{13}&=\{\lambda_3^{(2)}, \lambda_6^{(2)}, \lambda_7^{(2)}, \lambda_5^{(3)}, \lambda_6^{(3)}\},\qquad \mathfrak{D}_{14}=\{\lambda_3^{(2)}, \lambda_7^{(2)}, \lambda_1^{(3)}, \lambda_5^{(3)}, \lambda_6^{(3)}\},\\
\mathfrak{D}_{15}&=\{\lambda_6^{(2)}, \lambda_7^{(2)}, \lambda_4^{(3)}, \lambda_5^{(3)}, \lambda_6^{(3)}\},\qquad \mathfrak{D}_{16}=\{\lambda_7^{(2)}, \lambda_1^{(3)}, \lambda_4^{(3)}, \lambda_5^{(3)}, \lambda_6^{(3)}\}.
\end{align*}
From Theorem \ref{2023.11.26.14598}, for $1\le i\le 16$, we have $\sigma(u)\in\mathbb{Z}[{\boldsymbol \l}_{\mathfrak{D}_i}]\langle\langle u \rangle\rangle$. 
We have $\mathfrak{A}=\mathfrak{D}_6$. 
We cannot apply Proposition \ref{2022.5.29.1} to this curve. 
\end{ex}

\begin{ex}\label{2023.12.19.4}
We consider the case of $m=3$ and $A_3=(4,6,7)$. 
The polynomials $F_2$ and $F_3$ defining this curve are given in Example \ref{2022.5.29.2} (iii). 
We have $\mathfrak{A}=\{\lambda_1^{(2)}, \lambda_5^{(2)}, \lambda_1^{(3)}, \lambda_3^{(3)}, \lambda_7^{(3)}\}$. 
From Theorem \ref{maintheorem} (i), we have $\sigma(u)\in\mathbb{Z}[{\boldsymbol \l}_{\mathfrak{A}}]\langle\langle u \rangle\rangle$.  
By applying the operation (A), we can transform $\det G_3$ into a polynomial $\mathscr{M}$ in the form of (\ref{2023.12.3.1}).  
Here, we omit the explicit expression of $\mathscr{M}$.  
We have $\mathscr{M}=\mathscr{M}_1+\mathscr{M}_2$ with $\mathscr{M}_2\in\mathcal{P}(\mathbb{Z}[{\boldsymbol \l}])$ and 
\[
\mathscr{M}_1=(\lambda_1^{(2)}+\lambda_1^{(3)})X_1^2X_3+(\lambda_1^{(2)}\lambda_8^{(3)}+\lambda_2^{(2)}\lambda_7^{(3)}+\lambda_5^{(2)}\lambda_4^{(3)}+\lambda_6^{(2)}\lambda_3^{(3)}+\lambda_8^{(2)}\lambda_1^{(3)})X_3. 
\]
From Lemma \ref{2023.12.31.1111}, we have $\mathcal{S}(\mathscr{M})=\mathcal{S}(\mathscr{M}_1)$. 
Let 
\begin{align*}
\mathfrak{E}_1=\{\lambda_1^{(2)}, \lambda_2^{(2)}, \lambda_5^{(2)}, \lambda_6^{(2)}, \lambda_1^{(3)}\},\qquad \mathfrak{E}_2=\{\lambda_1^{(2)}, \lambda_2^{(2)}, \lambda_5^{(2)}, \lambda_1^{(3)}, \lambda_3^{(3)}\},\\
\mathfrak{E}_3=\{\lambda_1^{(2)}, \lambda_2^{(2)}, \lambda_6^{(2)}, \lambda_1^{(3)}, \lambda_4^{(3)}\},\qquad \mathfrak{E}_4=\{\lambda_1^{(2)}, \lambda_2^{(2)}, \lambda_1^{(3)}, \lambda_3^{(3)}, \lambda_4^{(3)}\},\\
\mathfrak{E}_5=\{\lambda_1^{(2)}, \lambda_5^{(2)}, \lambda_6^{(2)}, \lambda_1^{(3)}, \lambda_7^{(3)}\},\qquad \mathfrak{E}_6=\{\lambda_1^{(2)}, \lambda_5^{(2)}, \lambda_1^{(3)}, \lambda_3^{(3)}, \lambda_7^{(3)}\},\\
\mathfrak{E}_7=\{\lambda_1^{(2)}, \lambda_6^{(2)}, \lambda_1^{(3)}, \lambda_4^{(3)}, \lambda_7^{(3)}\},\qquad \mathfrak{E}_8=\{\lambda_1^{(2)}, \lambda_1^{(3)}, \lambda_3^{(3)}, \lambda_4^{(3)}, \lambda_7^{(3)}\}. 
\end{align*}
From Theorem \ref{2023.11.26.14598}, for $1\le i\le 8$, we have $\sigma(u)\in\mathbb{Z}[{\boldsymbol \l}_{\mathfrak{E}_i}]\langle\langle u \rangle\rangle$. 
We have $\mathfrak{A}=\mathfrak{E}_6$. 
We can apply Proposition \ref{2022.5.29.1} to this curve. 
We have $\mathfrak{B}=\mathfrak{E}_4$. 
\end{ex}

\begin{ex}\label{2023.12.19.5}
We consider the case of $m=3$ and $A_3=(6,9,5)$. 
The polynomials $F_2$ and $F_3$ defining this curve are given in Example \ref{2022.5.29.2} (iv). 
We have $\mathfrak{A}=\{\lambda_1^{(2)}, \lambda_3^{(2)}, \lambda_7^{(2)}, \lambda_9^{(2)}, \lambda_{13}^{(2)}, \lambda_1^{(3)}, \lambda_3^{(3)}, \lambda_5^{(3)}, \lambda_9^{(3)}, \lambda_{15}^{(3)}\}$. 
From Theorem \ref{maintheorem} (i), we have $\sigma(u)\in\mathbb{Z}[{\boldsymbol \l}_{\mathfrak{A}}]\langle\langle u \rangle\rangle$. 
By applying the operation (A), we can transform $\det G_2$ into a polynomial $\mathscr{N}$ in the form of (\ref{2023.12.3.1}). 
Here, we omit the explicit expression of $\mathscr{N}$.  
We have $\mathscr{N}=\mathscr{N}_1+\mathscr{N}_2$ with $\mathscr{N}_2\in\mathcal{P}(\mathbb{Z}[{\boldsymbol \l}])$ and 
\begin{align*}
\mathscr{N}_1&=(\lambda_1^{(2)}+\lambda_1^{(3)})X_1^2X_2+(\lambda_2^{(2)}\lambda_1^{(3)}+\lambda_3^{(2)})X_2X_3^2\\
&+(\lambda_3^{(2)}\lambda_{10}^{(3)}+\lambda_4^{(2)}\lambda_9^{(3)}+\lambda_7^{(2)}\lambda_6^{(3)}+\lambda_9^{(2)}\lambda_4^{(3)}
+\lambda_{12}^{(2)}\lambda_1^{(3)}+\lambda_{13}^{(2)})X_2. 
\end{align*}
From Lemma \ref{2023.12.31.1111}, we have $\mathcal{S}(\mathscr{N})=\mathcal{S}(\mathscr{N}_1)$. 
Let 
\begin{align*}
\mathfrak{F}_1&=\{\lambda_1^{(2)}, \lambda_3^{(2)}, \lambda_4^{(2)}, \lambda_7^{(2)}, \lambda_9^{(2)}, \lambda_{13}^{(2)}, \lambda_1^{(3)}\},
\qquad \mathfrak{F}_2=\{\lambda_1^{(2)}, \lambda_3^{(2)}, \lambda_4^{(2)}, \lambda_7^{(2)}, \lambda_{13}^{(2)}, \lambda_1^{(3)}, \lambda_4^{(3)}\},\\
\mathfrak{F}_3&=\{\lambda_1^{(2)}, \lambda_3^{(2)}, \lambda_4^{(2)}, \lambda_9^{(2)}, \lambda_{13}^{(2)}, \lambda_1^{(3)}, \lambda_6^{(3)}\},
\qquad \mathfrak{F}_4=\{\lambda_1^{(2)}, \lambda_3^{(2)}, \lambda_4^{(2)}, \lambda_{13}^{(2)}, \lambda_1^{(3)}, \lambda_4^{(3)}, \lambda_6^{(3)}\},\\
\mathfrak{F}_5&=\{\lambda_1^{(2)}, \lambda_3^{(2)}, \lambda_7^{(2)}, \lambda_9^{(2)}, \lambda_{13}^{(2)}, \lambda_1^{(3)}, \lambda_9^{(3)}\}
,\qquad \mathfrak{F}_6=\{\lambda_1^{(2)}, \lambda_3^{(2)}, \lambda_7^{(2)}, \lambda_{13}^{(2)}, \lambda_1^{(3)}, \lambda_4^{(3)}, \lambda_9^{(3)}\},\\
\mathfrak{F}_7&=\{\lambda_1^{(2)}, \lambda_3^{(2)}, \lambda_9^{(2)}, \lambda_{13}^{(2)}, \lambda_1^{(3)}, \lambda_6^{(3)}, \lambda_9^{(3)}\}
,\qquad \mathfrak{F}_8=\{\lambda_1^{(2)}, \lambda_3^{(2)}, \lambda_{13}^{(2)}, \lambda_1^{(3)}, \lambda_4^{(3)}, \lambda_6^{(3)}, \lambda_9^{(3)}\}.
\end{align*}
From Theorem \ref{2023.11.26.14598}, for $1\le i\le 8$, we have $\sigma(u)\in\mathbb{Z}[{\boldsymbol \l}_{\mathfrak{F}_i}]\langle\langle u \rangle\rangle$. 
We have $\mathfrak{F}_5\subsetneq\mathfrak{A}$. 
By applying the operation (A), we can transform $\det G_3$ into a polynomial $\mathscr{O}$ in the form of  (\ref{2023.12.3.1}).  
Here, we omit the explicit expression of $\mathscr{O}$.  
We have $\mathscr{O}=\mathscr{O}_1+\mathscr{O}_2$ with $\mathscr{O}_2\in\mathcal{P}(\mathbb{Z}[{\boldsymbol \l}])$ and 
\begin{align*}
\mathscr{O}_1&=(\lambda_1^{(2)}+\lambda_1^{(3)})X_1^2X_3+(\lambda_2^{(2)}\lambda_1^{(3)}+\lambda_3^{(2)})X_1X_2+\left(\lambda_3^{(2)}\lambda_1^{(3)}+\lambda_2^{(2)}(\lambda_1^{(3)})^2\right)X_2X_3\\
&+(\lambda_3^{(2)}\lambda_4^{(3)}+\lambda_2^{(2)}\lambda_1^{(3)}\lambda_4^{(3)})X_1X_3+(\lambda_3^{(2)}\lambda_6^{(3)}+\lambda_2^{(2)}\lambda_1^{(3)}\lambda_6^{(3)})X_2+(\lambda_3^{(2)}\lambda_9^{(3)}+\lambda_2^{(2)}\lambda_1^{(3)}\lambda_9^{(3)})X_1\\
&+(\lambda_4^{(2)}\lambda_9^{(3)}+\lambda_7^{(2)}\lambda_6^{(3)}+\lambda_9^{(2)}\lambda_4^{(3)}+\lambda_{12}^{(2)}\lambda_1^{(3)}+\lambda_{13}^{(2)}+\lambda_2^{(2)}\lambda_1^{(3)}\lambda_{10}^{(3)})X_3. 
\end{align*}
We obtain the same result as the case of $\det G_2$. 
We cannot apply Proposition \ref{2022.5.29.1} to this curve. 

\end{ex}

\begin{rem}
For the above three curves, Theorem \ref{2023.11.26.14598} gives the better result than Theorem \ref{maintheorem} (i). 
\end{rem}

\vspace{2ex}

{\bf Acknowledgements.} 
This work was supported by JSPS KAKENHI Grant Number JP21K03296 and was partly supported by MEXT Promotion of Distinctive Joint Research Center Program JPMXP0723833165.  



\begin{thebibliography}{9}

\bibitem{Aya1}
T. Ayano, Sigma functions for telescopic curves, Osaka J. Math., \textbf{51} (2014), 459--480.

\bibitem{Aya3}
T. Ayano, On Jacobi Inversion Formulae for Telescopic Curves, SIGMA, \textbf{12} (2016), 086, 21 pages. 

\bibitem{AB2020}
T. Ayano and V. M. Buchstaber, Analytical and number-theoretical properties of the two-dimensional sigma function, Chebyshevskii Sb., \textbf{21} (2020), 9--50. 

\bibitem{Aya2}
T. Ayano and A. Nakayashiki, On Addition Formulae for Sigma Functions of Telescopic Curves, SIGMA, \textbf{9} (2013), 046, 14 pages. 

\bibitem{Baker2}
H. F. Baker, An Introduction to the Theory of Multiply Periodic Functions, Cambridge University Press, 1907.

\bibitem{Baker}
H. F. Baker, Abelian Functions. Abel's theorem and the allied theory of theta functions, Cambridge University Press, 1995. 

\bibitem{BEN2020}
J. Bernatska, V. Enolski and A. Nakayashiki, Sato Grassmannian and Degenerate Sigma Function, Commun. Math. Phys., \textbf{374} (2020), 627--660. 

\bibitem{BertinCarbonne} 
J. Bertin and P. Carbonne, Semi-Groupes d'Entiers et Application aux Branches, J. Algebra, \textbf{49} (1977), 81--95. 

\bibitem{Bolza0}
O. Bolza, Proof of Brioschi's Recursion Formula for the Expansion of the Even $\sigma$-Functions of Two Variables, Amer. J. Math., \textbf{21} (1899), 175--190. 

\bibitem{Bolza}
O. Bolza, Remarks Concerning the Expansions of the Hyperelliptic Sigma-Functions, Amer. J. Math., \textbf{22} (1900), 101--112. 

\bibitem{BEL-97-1}    
V. M. Buchstaber, V. Z. Enolskii and D. V. Leykin, Hyperelliptic Kleinian Functions and Applications, in Solitons, Geometry, and Topology: On the Crossroad, Amer. Math. Soc. Transl. Ser. 2, \textbf{179}, Amer. Math. Soc., Providence, RI, 1997, 1–33. 


\bibitem{BEL-97-2}
V. M. Buchstaber, V. Z. Enolskii and D. V. Leykin, Hyperelliptic Abelian Functions, 1997, available at 

\begin{verbatim}
https://www.researchgate.net/publication/266955336_Kleinian_functions
_hyperelliptic_Jacobians_and_applications
\end{verbatim}


\bibitem{BEL-99-R}  
V. M. Buchstaber, V. Z. Enolskii and D. V. Leykin, Rational Analogs of Abelian Functions, Funct. Anal. Appl., \textbf{33} (1999), 83--94. 


\bibitem{BEL-2000}
V. M. Buchstaber, V. Z. Enolskii and D. V. Leykin, Uniformization of Jacobi Varieties of Trigonal Curves and Nonlinear Differential Equations, Funct. Anal. Appl., \textbf{34} (2000), 159--171. 

\bibitem{BEL-2012}  
V. M. Buchstaber, V. Z. Enolski and D. V. Leykin, Multi-Dimensional Sigma-Functions, arXiv:1208.0990, (2012). 

\bibitem{BEL-2018}
V. M. Buchstaber, V. Z. Enolski and D. V. Leykin, $\sigma$-Functions: Old and New Results, In R. Donagi and T. Shaska (Eds.), Integrable Systems and Algebraic Geometry (London Mathematical Society Lecture Note Series 459), Cambridge University Press, 
2020, 175--214.  


\bibitem{BL-2005}
V. M. Buchstaber and D. V. Leykin, Addition Laws on Jacobian Varieties of Plane Algebraic Curves, Proc. Steklov Inst. Math., \textbf{251} (2005), 49--120.


\bibitem{BEL-99-2}
V. M. Buchstaber, D. V. Leykin and V. Z. Enolskii, $\sigma$-functions of $(n,s)$-curves, Russ. Math. Surv., \textbf{54} (1999), 628--629. 

\bibitem{Bunkova1}
E. Yu. Bunkova, Weierstrass Sigma Function Coefficients Divisibility Hypothesis, arXiv:1701.00848, (2017).

\bibitem{CN}
K. Cho and A. Nakayashiki, Differential Structure of Abelian Functions, Int. J. Math., \textbf{19} (2008), 145--171. 

\bibitem{EEL}
J. C. Eilbeck, V. Z. Enolskii and D. V. Leykin, On the Kleinian Construction of Abelian Functions of Canonical Algebraic Curves, In proceedings of the Conference SIDE III: 
Symmetries and Integrability of Difference Equations (Sabaudia, 1998), CRM Proc. Lecture Notes, \textbf{25}, 
Amer. Math. Soc., Providence, RI, 2000, 121--138. 

\bibitem{O5}
J. C. Eilbeck, J. Gibbons, Y. \^Onishi and S. Yasuda, Theory of Heat Equations for Sigma Functions, arXiv:1711.08395, (2017). 

\bibitem{EO2019}
J. C. Eilbeck and Y. \^Onishi, Recursion Relations on the Power Series Expansion of the Universal Weierstrass Sigma Function, RIMS K\^oky\^uroku Bessatsu, \textbf{B78} (2020), 077--098. 

\bibitem{Fay-1973}
J. D. Fay, Theta Functions on Riemann Surfaces, Lecture Notes in Math., \textbf{352}, Springer Berlin, Heidelberg, 1973.

\bibitem{KP}
C. Kirfel and R. Pellikaan, The minimum distance of codes in an array coming from telescopic semigroups, IEEE Trans. Inform. Theory, \textbf{41} (1995), 1720--1732.

\bibitem{Kl1}
F. Klein, Ueber hyperelliptische Sigmafunctionen, Math. Ann., \textbf{27} (1886), 431--464.

\bibitem{Kl2}
F. Klein, Ueber hyperelliptische Sigmafunctionen, Math. Ann., \textbf{32} (1888), 351--380.

\bibitem{Miu}
S. Miura, Linear Codes on Affine Algebraic Curves, IEICE Trans., \textbf{J81-A} (1998), 1398--1421 (in Japanese).


\bibitem{N1}   
A. Nakayashiki, On Algebraic Expressions of Sigma Functions for $(n,s)$ Curves, Asian J. Math., \textbf{14} (2010), 175--212. 

\bibitem{N-2010-2}
A. Nakayashiki, Sigma Function as A Tau Function,  Int. Math. Res. Not., \textbf{2010} (2010), 373--394.

\bibitem{N-2011}
A. Nakayashiki, On hyperelliptic abelian functions of genus 3, J. Geom. Phys., \textbf{61} (2011), 961--985. 

\bibitem{N-2018}
A. Nakayashiki, Degeneration of trigonal curves and solutions of the KP-hierarchy, Nonlinearity, \textbf{31} (2018), 3567–3590.  

\bibitem{NW}
A. Nijenhuis and H. S. Wilf, Representations of Integers by Linear Forms in Nonnegative Integers, J. Number Theory, \textbf{4} (1972),  98--106. 

\bibitem{Ouniversal}
Y. \^Onishi, Universal Elliptic Functions, (in Japanese), 2016, available at

\begin{verbatim}
https://math-meijo-u.heteml.net/#publications
\end{verbatim}

\bibitem{O-2018}
Y. \^Onishi, Arithmetical Power Series Expansion of the Sigma Function for a Plane Curve, Proc. Edinburgh Math. Soc., \textbf{61} (2018), 995--1022. 

\bibitem{MSDCS}
M. Schiffer and D. C. Spencer, Functionals of Finite Riemann Surfaces, Princeton University Press, 1954. 




\end{thebibliography}
\end{document}